\documentclass[leqno,12pt]{article}
\usepackage[margin=1.3in, top = 1.3in, bottom=1.3in]{geometry}
\usepackage[utf8]{inputenc}
\usepackage{amsmath}
\usepackage{amsfonts}
\usepackage{mathrsfs}
\usepackage{tikz}
\usepackage{hyperref}
\usepackage{stmaryrd}
\usepackage{enumerate}
\usepackage{tikz-qtree}
\usetikzlibrary{automata,positioning}
\usepackage{amssymb}
\usepackage{amsthm}
\usepackage{bbm}
\usepackage{mathtools}
\usepackage{algorithm}
\usepackage{algpseudocode}
\usepackage{diffcoeff}
\usepackage{interval}
\usepackage{graphicx}
\usepackage{comment}
\usepackage{fancyhdr}
\usepackage{tikz}
\usepackage{tikz-3dplot}
\usepackage{enumitem}
\usetikzlibrary{shapes,arrows}
\DeclarePairedDelimiter\ceil{\lceil}{\rceil}

\theoremstyle{definition}
\newcommand{\R}{\mathbb{R}}
\newcommand{\st}[1]{\left\{#1\right\}}
\newcommand{\ip}[1]{\left\langle#1\right\rangle}

\newcommand{\N}{\mathbb{N}}

\newcommand{\E}{\mathbb{E}}

\newcommand{\grad}{\nabla}
\newcommand{\from}{\colon}
\newcommand{\seq}{\subseteq}

\newcommand{\clos}{\overline}

\newcommand{\aaa}{\mathcal{A}}
\newcommand{\fami}{\mathcal{F}}

\newcommand{\eps}{\varepsilon}

\newcommand{\p}{\partial}

\DeclareMathOperator{\inter}{int}

\DeclareMathOperator*{\argmin}{argmin}

\DeclareMathOperator{\dist}{dist}

\DeclareMathOperator{\epi}{epi}

\DeclareMathOperator{\arccosh}{arccosh}

\newtheorem{definition}{Definition}[section]
\newtheorem{theorem}{Theorem}[section]
\newtheorem{corollary}{Corollary}[section]
\newtheorem{proposition}{Proposition}[section]
\newtheorem{lemma}[theorem]{Lemma}

\newtheorem{example}[theorem]{Example}
\newtheorem{remark}[theorem]{Remark}

\newtheorem{innercustomthm}{Assumption}
\newenvironment{customthm}[1]
  {\renewcommand\theinnercustomthm{#1}\innercustomthm}
  {\endinnercustomthm}

\tikzstyle{decision} = [diamond, draw, fill=blue!20, 
text width=6em, text badly centered, node distance=3cm, inner sep=0pt]
\tikzstyle{block} = [rectangle, draw, fill=blue!20, 
text width=9em, text centered, rounded corners, minimum height=4em]
\tikzstyle{block2} = [rectangle, draw, fill=green!20, 
text width=9em, text centered, rounded corners, minimum height=4em]
\tikzstyle{line} = [draw, -latex']
\tikzstyle{cloud} = [draw, ellipse,fill=red!20, node distance=3cm,
minimum height=4em]

\counterwithin{figure}{section}

\begin{document}
\title{\vspace{-2.5cm} Stochastic and incremental subgradient methods for convex optimization on Hadamard spaces}
\author{Ariel Goodwin
\thanks{CAM, Cornell University, Ithaca, NY.
	\texttt{awg77@cornell.edu}. Research supported in part by the NSERC Postgraduate Fellowship PGSD-587671-2024.
}
\and
Adrian S. Lewis
\thanks{ORIE, Cornell University, Ithaca, NY.
\texttt{adrian.lewis@cornell.edu} Corresponding Author.
Research supported in part by National Science Foundation Grant DMS-2405685.}
\and
Genaro L\'opez-Acedo
\thanks{Department of Mathematical Analysis -- IMUS, University of Seville, 41012 Seville, Spain
\texttt{glopez@us.es}. Research supported in part by DGES, project PID2023-148294NB-I00.
}
\and
Adriana Nicolae\thanks{Department of Mathematics, Babe\c{s}-Bolyai University, 400084, Cluj-Napoca, Romania \hfill \mbox{}
\texttt{anicolae@math.ubbcluj.ro}.
Research supported in part by the Ministry of Research, Innovation and Digitization, CNCS/CCCDI -- UEFISCDI, project number PN-III-P1-1.1-TE-2019-1306, within PNCDI III.
}
}
\date{\today}

	\maketitle

\begin{abstract}
As a foundation for optimization, convexity is useful beyond the classical settings of Euclidean and Hilbert space.  The broader arena of nonpositively curved metric spaces, which includes manifolds like hyperbolic space, as well as metric trees and more general CAT(0) cubical complexes, supports primal tools like proximal operations for geodesically convex functions.  However, the lack of linear structure in such spaces complicates dual constructions like subgradients.  To address this hurdle, we introduce a new type of subgradient for functions on Hadamard spaces, based on Busemann functions.  Our notion supports generalizations of classical stochastic and incremental subgradient methods, with guaranteed complexity bounds.  We illustrate with subgradient algorithms for $p$-mean problems in general Hadamard spaces, in particular computing medians in BHV tree space.

\end{abstract}
\medskip

\medskip


\begin{section}{Introduction}
We consider optimization problems posed over a subset $C$ of a complete metric space $(X,d)$.  Beginning our discussion informally, we assume that $X$ is a {\em Hadamard space}, meaning that it has nonpositive curvature --- the ``CAT(0)'' property from metric geometry \cite{bridson}.  This framework covers the familiar example of Hilbert space, and all complete simply-connected Riemannian manifolds of nonpositive sectional curvature.  Examples of such manifolds include Euclidean and hyperbolic spaces, and spaces of positive-definite symmetric matrices with the affine-invariant metric \cite{boumal2023intromanifolds}.  However, the Hadamard space framework also subsumes interesting examples that are not manifolds, such as the Billera-Holmes-Vogtmann (BHV) tree space \cite{billeratree} and, more generally, all CAT(0) cubical complexes.  A simple example to keep in mind is the {\em tripod}, which consists of three copies of the halfline $\R_+ = [0,+\infty)$ with its usual metric, glued together by identifying the three copies of 0. The fundamentals of convex optimization on
Hadamard space are treated extensively in \cite{bacakbook}.

We assume that the Hadamard space $X$ has the {\em extension property\/}.  In other words, given any two points $x$ and $y$ in $X$, there exists an isometry from $\R_+$ into $X$ mapping $0$ to $x$ and the distance $d(x,y)$ to $y$.  This isometry is a {\em geodesic ray} from $x$ through $y$:  there might be several such rays, but the corresponding unit-speed path joining $x$ to $y$ (the {\em geodesic}) is unique.  For algorithmic purposes, we assume that we can readily compute such a ray, along with $d(x,y)$.  Such is the case for each of the examples we have mentioned (\cite{boumal2023intromanifolds,hayashi}).
	
For tractability, as in the classical Euclidean case, we focus on problems
\begin{equation}
\label{eqn:basic}
\inf_C f
\end{equation}
where both the feasible region $C$ and the objective $f \colon C \to \R$ are convex in the natural geodesic sense.  In general Hadamard spaces, as shown in \cite{bacak-means}, such optimization problems are solvable in principle by iterating the fundamental proximal update
\begin{equation}
\label{eqn:proximal-update}
x ~\leftarrow~ \mbox{arg}\!\min \{ f(y) + \alpha d(x,y)^2 : y \in C \}
\end{equation}
for a parameter $\alpha > 0$.  This update is sometimes easy to implement.  In the particular case
\begin{equation}
\label{eqn:distance-power}
f(x) = d^p(x,a), \qquad x \in C = X
\end{equation}
for some point $a \in X$ and exponent $p \ge 1$, the update (\ref{eqn:proximal-update}) lies on the geodesic between $a$ and $x$, with a simple closed form in the cases $p=1,2$.  In general, however, the proximal update has several disadvantages as an optimization tool.  Even in the simple case 
(\ref{eqn:distance-power}), exponents $p \ne 1,2$ entail bisection search along the geodesic, and generalizing beyond the case $C=X$ is difficult.  Moreover, the relationship between the proximal parameter $\alpha$ and the length of the update step is obscure.

Following standard optimization practice, we refine our problem (\ref{eqn:basic}) to allow for additive structure:
\begin{equation}
\label{eqn:structured}
f(x) ~=~ \sum_{i=1}^m f_i(x) \qquad (x \in C),
\end{equation}
where the component functions $f_i \colon C \to \R$ are geodesically convex.  Rather than iterating the proximal update for $f$, we can then apply a splitting approach as analyzed in \cite{bacak-means}, either stochastically choosing one component $f_i$ per iteration or cyclically applying proximal updates to each component $f_i$ in turn, with strategically chosen proximal parameters.  The convergence rate of the resulting iterates appears in \cite{OhtaPalfia2015}, under additional strong convexity conditions.  In general, as we later discuss, the analyses of \cite{bacak-means,OhtaPalfia2015} implicitly reveal that $O(\eps^{-2})$ iterations can suffice to approximate the optimal value within a tolerance $\eps>0$.
Proximal splitting provides a tractable approach to {\em $p$-mean problems} (for $p \ge 1$), where each function $f_i$ has the form (\ref{eqn:distance-power}).  The case $p=1$ is the classical \textit{Weber problem} \cite{weber} in location theory.

Since proximal updates are rarely implementable, even in the classical Euclidean case we 
must often rely instead on subgradient-based methods.  We therefore consider the following question:  
\begin{quote}
{\em Do classical subgradient methods extend to general Hadamard spaces?}   
\end{quote}
This question presents an immediate conundrum.  Whereas the proximal iteration (\ref{eqn:proximal-update}) is primal in nature, simply defined on the underlying Hadamard space $X$, subgradients in the Euclidean setting are dual objects --- linear functionals on $X$.  In the absence of any linear structure, the analogue in Hadamard space is unclear.

To handle this fundamental hurdle when the space $X$ is a manifold, the subgradient methods introduced in \cite{riemsub} and the complexity analysis of \cite{geoconvex} rely on local linearization, much like the theory of smooth optimization in \cite{boumal2023intromanifolds}.  These works view gradients and subgradients at a point $x \in X$ as elements of the tangent space $T_x X$, forcing use of the exponential map $\mbox{Exp}_x \colon T_x X \to X$ or some approximate version.  Moreover, the complexity bound in \cite{geoconvex} depends unavoidably on a lower curvature bound for $X$ \cite{bettercurv}.  

This approach is technical, and furthermore fails in general Hadamard spaces $X$ because the machinery of local linearization, duality, and lower curvature bounds is unavailable.  We argue here for an entirely different approach, one that is simpler, global, primal, and requires no lower curvature bound.  We identify subgradients with constant-speed geodesic rays in $X$, and we escape the curvature-related worst-case bounds of \cite{bettercurv} by restricting the class of objectives $f$, arriving at an algorithm with complexity analogous to the Euclidean case. 

A first step in this direction was the recent development of a \textit{horospherical subgradient algorithm} in \cite{lewis2024horoballs}.  This method applies in a general Hadamard space $X$, but requires a quasiconvexity property:  at any point $x \in C$, the level set
\[
L_x ~=~ \{y \in C : f(y) \le f(x) \}
\]
must be horospherically convex.  In non-Euclidean settings, horospherical quasiconvexity is a significant assumption,  but one that holds in a variety of interesting cases.  The algorithm relies on an oracle that returns a geodesic ray originating from $x$ and that ``supports'' $L_x$ in a certain horospherical sense.  In Euclidean space, rays opposite to normal vectors for $L_x$ suffice, mimicking the key property of subgradient directions, but in general the supporting property is more restrictive.  The algorithm then takes a step from $x$ along this supporting ray.  The resulting complexity parallels the Euclidean case.  

Horospherical ideas had appeared earlier in optimization, in works such as \cite{fenchelhadam,fan2023horospherical,convanalyshad}.  Horospheres in Hadamard spaces are limits of spheres:  in Euclidean space, they are hyperplanes.  Horospheres have the form 
$\{x \in X : b(x)=0 \}$ for {\em Busemann functions} $b$, the natural generalizations of affine functions on Euclidean spaces. Building on these analogies, various notions of Fenchel conjugation based on Busemann functions have appeared in the recent literature \cite{fenchelhadam,convanalyshad}.

Although the supporting ray oracle in \cite{lewis2024horoballs} is implementable for some interesting functions, the horospherical subgradient algorithm is difficult to apply in much generality, because unlike Euclidean subgradients, supporting rays have no obvious calculus.  Specifically, for structured objectives $f = \sum_i f_i$, we cannot easily combine supporting rays for each component $f_i$ into a supporting ray for $f$.  We therefore turn our attention instead to stochastic and incremental subgradient methods, appealing simple algorithms that employ steps computed only from individual components $f_i$. 

Our approach is inspired by the classical Euclidean complexity analysis of the stochastic subgradient method, as well as the incremental subgradient algorithm of \cite{bertseknedic}.
Quasiconvex versions of the latter algorithm appear in \cite{quasiincr, normincr}, and were considered on manifolds with curvature bounded below in \cite{manifoldquasi}, but these algorithms have the drawback that all the components $f_i$ must share a common minimizer.  This restriction rules out many interesting examples, including the $p$-mean problem, but it seems inherent to methods relying only on supporting ray oracles, which cannot distinguish between different component functions having geometrically similar level sets.  This drawback suggests the need for a stronger oracle: supporting rays give directional information but that alone does not suffice.  

The new stochastic and incremental algorithms we introduce here rely instead on a {\em Busemann subgradient} oracle. For each component function $f_i$ and point $x$ in $C$, the oracle returns a Busemann function $b$ and a ``speed'' $s \ge 0$ satisfying a subgradient-like inequality 
\[
f_i(y) - f_i(x) ~\ge~ s\big(b(y) - b(x) \big) \qquad \mbox{for all}~ y \in C.
\]
The algorithm then tracks the geodesic ray from $x$ associated with $b$ at speed $s$ for a judiciously chosen time interval, before repeating the process.
Analogous subgradient inequalities appeared in \cite{AHMADIKAKAVANDI,sumrule}, and indeed, the notion of a Busemann subgradient that we introduce is related to ideas in \cite{GigliNobili2021,sumrule}.   The algorithms we present generalize the classical stochastic subgradient method, when choosing components uniformly at random, and the incremental subgradient algorithm of \cite{bertseknedic}:  furthermore, they  enjoy the same complexity. While the $p$-mean problem serves as a familiar testing ground for our algorithms, we emphasize that Busemann subgradient ideas apply more broadly:  examples include the minimum enclosing ball problem~\cite{arnaudon2013}, and, as highlighted recently in \cite{criscitiello2025}\footnote{While revising this paper from the original posted manuscript \cite{splitting-subgradient}, we learned of subsequent developments involving the Busemann subgradient idea \cite{criscitiello2025}.}, Tyler's M-estimator \cite{tyler1987} and Horn's problem on deciding if given real vectors are the spectra of Hermitian matrices summing to zero \cite{Knutson2001}, \cite[p.\ 1]{kapovich2009}\footnote{A concrete formulation of Horn's problem in terms of
		Busemann functions appears in \cite[p.~19]{criscitiello2025},  an approach originating  
from the work of \cite{kapovich2009}.}. %

The structure of this paper is as follows. We begin by reviewing some metric geometry, and consequences of nonpositive curvature. Central to our development is the concept of \textit{direction}, so we review some large-scale geometric properties of Hadamard spaces.  This exploration leads to a new property, called \textit{Busemann subdifferentiability}, strong enough to support a useful notion of subgradients in nonlinear space.  We compare this new notion with earlier ideas about subgradients in Hadamard space, and relate it to horospherical convexity.  Remarkably, even in the Euclidean case, Busemann subdifferentiability gives a new perspective on subgradients of convex functions, allowing a fundamentally geometric or primal understanding, without explicit reference to the inner product. We show that many natural functions defined in terms of metric data are Busemann subdifferentiable, and we explore the calculus of Busemann subgradients. 

Rather than formal convex analysis, however, our aim is the development of implementable algorithms with complexity guarantees for structured convex optimization.  Notably, we remark that Busemann subdifferentiability is {\em not} preserved under addition, explaining the crucial use of splitting in Busemann subgradient-based algorithms.  The algorithms we develop enjoy complexity analyses matching their Euclidean analogues, as well as the proximal algorithms appearing in \cite{bacak-means}.  Specifically, we prove that we can approximate the optimal value (in expectation, for the stochastic method)  within a tolerance $\eps$ using $O(\eps^{-2})$ iterations.  To illustrate the approach computationally, we solve the Weber (1-mean) problem in BHV tree space \cite{billeratree}, for some small examples from \cite{sturmmean}. 	
	\end{section}

		 \begin{section}{Geodesic geometry and convexity}
		\label{sec:buse}
A metric space $(X,d)$ is a
		\textit{geodesic metric space} if every two points $x,y\in X$ can be joined by a 
		 \textit{geodesic},
		which is to say a map $\gamma$ from a closed interval $[a,b]$ into $X$ with $d(\gamma(t),\gamma(t')) = |t-t'|$
		for all $t,t'\in [a,b]$. A \textit{ray} is a geodesic with domain $\R_+$, and we say $r$ \textit{issues from}
		$x$ if $r(0) = x$. A geodesic metric space is said to be \textit{CAT(0)} if the map $t\mapsto \frac{1}{2}d(\gamma(t),a)^2$ is $1$-strongly convex for every $a\in X$ and geodesic $\gamma$. A \textit{Hadamard space} is a complete CAT(0) space. 
		In a Hadamard space, there is exactly one geodesic joining $x$ to $y$ for every $x,y\in X$. 
		We say that
		a Hadamard space has the \textit{geodesic extension property} if for every $x\neq y\in X$ there exists a 
		ray $r\from\R_+\to X$ with $r(0) = x$ and $r(t) = y$ for some $t>0$. A subset $C$ of a Hadamard 
		space $X$ is \textit{geodesically
		convex} if the geodesic between any two points in $C$ is contained in $C$. A metric space $(X,d)$ is \textit{proper} if the closed ball $B_r(x)$ is compact for every $x\in X, r > 0$.

		Henceforth, $(X,d)$ will be a Hadamard space with the geodesic extension property, and to avoid degenerate
		case-splitting in forthcoming results we will assume $X$ has at least two points. The significance of this technical
		assumption is that there always exist rays.
		Given a ray $r$ on $X$ we may associate the 
		corresponding \textit{Busemann function} $b_r\from X\to \R$ defined by
		\[b_r(z) = \lim_{t\to\infty}(d(z,r(t)) - t) ~ ~ ~ (z\in X).\]
It is easy to see that the Busemann function $b_r$ is 1-Lipschitz, convex, and satisfies $b_r(r(0)) = 0$.
		Given a ray $r$, sets of the form $b_r^{-1}((-\infty,0])$
		are called \textit{horoballs}.

			We will restrict our attention to a class of functions that interacts
		nicely with the geometry of the given Hadamard space in a way that goes beyond geodesic convexity. 
		Fundamental to our development is an appropriate notion of \textit{direction} in a Hadamard space; in the
		Euclidean setting, the role played by directions and the associated
		compactification of $\R^n$ has been emphasized as a pillar of modern
		variational analysis \cite[Chapter 3]{rockafellar}.
		
\subsection*{The boundary $X^{\infty}$}
		We first review some standard tools to study the geometry at infinity of a Hadamard 
		space, a more detailed discussion of which can be found in
		\cite[Chapter II.8]{bridson}.
		Two rays $r,r'$ are said to be \textit{asymptotic}
		if there exists a positive constant $K\geq 0$ such that $d(r(t),r'(t)) \leq K$ for all $t\geq 0$. 
		This defines an equivalence relation on rays in $X$: the set of equivalence classes is 
		denoted by $X^{\infty}$ and called the \textit{boundary of $X$ at infinity}, or simply the \textit{boundary}. Note our assumption that $X$ has at least two points and the geodesic extension property guarantees the boundary is nonempty.
		The equivalence class of a particular 
		ray $r$ is denoted by $r(\infty)$, and we say a ray $r$ has \textit{direction} $\xi \in X^\infty$ if $r$
		belongs to the equivalence class $\xi$. Given $x\in X$ and $\xi \in X^\infty$, there exists a unique
		ray $r$ issuing from $x$ such that $r(\infty) = \xi$ \cite[II.8.2]{bridson}. 
		If we fix an arbitrary 
		reference point $\bar x \in X$, we can thus identify any $\xi \in X^\infty$ with the unique ray $r_{\bar x,\xi}$ 
		issuing from $\bar x$ with direction $\xi$, and hence to the unique
		corresponding Busemann function which we may denote by either $b_{r_{\bar x,\xi}}$ or $b_{\bar x,\xi}$. 
		To avoid cumbersome notation, we henceforth
		fix such a reference point $\bar x \in X$ and denote the corresponding Busemann function 
		$b_{r_{\bar x,\xi}}$ by simply $b_\xi$.  We note the following fact \cite[Corollary II.8.20]{bridson}.
		
\begin{proposition}
\label{additive}
The Busemann function corresponding to any fixed direction depends on the reference point only through an additive constant.
\end{proposition}
		
		The space $X^\infty$ is naturally endowed with the so-called \textit{cone topology} 
		\cite[Definition II.8.6]{bridson}, and notably this space is first-countable 
		\cite[Proof of Theorem II.8.13]{bridson}. 
				In particular, the cone topology is completely specified by 
		the convergent sequences, so let us mention that $\xi_n \to \xi \in X^\infty$ if and only if 
		$b_{\xi_n} \to b_{\xi}$ uniformly on bounded subsets of $X$.  The latter property, and hence the cone topology, is independent of the implicit basepoint (cf.\ \cite[Proposition II.8.8]{bridson}), due to Proposition \ref{additive}.

 If $X$ is proper then $X^\infty$ is compact 
		\cite[Definition II.8.6]{bridson}. Our interest in topologizing $X^\infty$ comes from optimization: it will be
		desirable to have a large class of compact subsets of $X^\infty$ so that certain continuous functions defined on 
		$X^\infty$ will attain their maximum. 
		In the literature, the set $X^\infty$ is often endowed with a stronger metric topology 
		induced by an angular metric (see for example \cite[Proposition II.9.7]{bridson}, \cite{convanalyshad}).
		This angular metric
		is convenient for the study of large-scale geometry of the space $X$ but is often too strong to be useful for 
		our purposes. For example, the boundary at infinity of 
		the $n$-dimensional hyperbolic
		space $\mathbb{H}^n$ is discrete in this angular metric, whereas it is homeomorphic to $\mathbb{S}^{n-1}$ in the
		 cone topology (see Example \ref{ex:conepoincare}). 
		In particular, the only compact sets are finite when the angular metric is used. Henceforth we will only
		make use of the cone topology.

		Convergence in $X^\infty$ can be understood more geometrically via
		pointwise convergence of rays, a fact stated in \cite[p.7]{duchesne2016} without proof, and also noted in \cite[p.29]{Ballmann1995Lectures}.
		The next proposition is a precise formulation:  for completeness, we include a proof in the appendix.
		The proof and result resemble \cite[Proposition II.8.19]{bridson}, but that result
concerns unbounded sequences in $X$ that converge to points in $X^\infty$ whereas we emphasize sequences 
		converging in $X^\infty$ itself.
		\begin{proposition}[\bf Convergence of directions]
			\label{prop:boundconv}
			Let $X$ be a Hadamard space with boundary $X^\infty$, let $\st{\xi_n}_{n=1}^\infty\seq X^\infty$ be a sequence, and
			let $\xi\in X^\infty$. For each $n=1,2,3,\ldots$, denote the ray issuing from $\bar x \in X$ with direction $\xi_n$ by
			$r_n$, and let $r$ be the ray issuing from $\bar x$ with direction $\xi$.
			Then $\xi_n\to \xi$ in $X^\infty$ if and only if $r_n(\delta)\to r(\delta)$ for all $\delta > 0$.
		\end{proposition}
		
\begin{example}[\bf Boundary of Euclidean space]
For the space $X = \R^n$, Proposition \ref{prop:boundconv} shows that the boundary $X^\infty$ is homeomorphic to the sphere $\mathbb{S}^{n-1}$.
\end{example}

		\begin{example}[\bf Boundary of the tripod]
			\label{ex:tripod}
			The \textit{tripod} is the Hadamard space $X$
			consisting of three copies of the half-line $\R_+$ glued together at the common point $0$.
			A natural choice for the reference point $\bar x$ is this common origin.
			Two rays are asymptotic if and only if they eventually lie in the same copy of $\R_+$, so 
			the boundary $X^\infty$ comprises three equivalence classes:
			$X^\infty = \st{\xi_1,\xi_2,\xi_3}$. Denoting a point in $X$ by $(x,j)$, where $x \in \R_+, j\in 
			\st{1,2,3}$ we can write the Busemann functions explicitly:
			\[b_{\xi_i}(x,j) = \begin{cases}
					x, & i \neq j\\
					-x, & i = j.
				\end{cases}
			\]
			Proposition \ref{prop:boundconv} tells us that a sequence of points in $X^\infty$ converges
			to a limit $\zeta\in X^\infty$
			if and only if the corresponding rays issuing from the origin converge pointwise to the ray defined by 
			$\zeta$. Clearly a sequence of rays in the tripod can only converge pointwise
			if it is eventually constant, i.e.\ the only convergent sequences in $X^\infty$ are eventually constant.
			Thus the boundary $X^\infty = \st{\xi_1,\xi_2,\xi_3}$ has the discrete topology. 
		\end{example}

\begin{example}[\bf Boundary of hyperbolic space]
			\label{ex:conepoincare}
			In this example we characterize the cone topology for the boundary of
			hyperbolic space $X = \mathbb{H}^n$. We use the Poincar\'{e} ball model, viewing $X$
			as the open unit ball in $\R^n$ with
			metric $d(p,q) = \arccosh\left(1 + 2\frac{\|p-q\|^2}{(1-\|p\|^2)(1-\|q\|^2)}\right)$.
			Proposition \ref{prop:boundconv}
			says a sequence
			$\st{\xi_n}_{n=1}^\infty \seq X^\infty$ converges to $\xi \in X^\infty$ if and only if 
			the corresponding rays $\st{r_n}_{n=1}^\infty$ issuing from $\bar x = 0 \in \mathbb{H}^n$ converge
			pointwise to the ray defined by $\xi$. Such rays 
			are radial lines from the origin to the corresponding boundary point $\xi_n$ ---we have once again
			identified $X^\infty$ with the boundary sphere $\mathbb{S}^{n-1}$.
			Explicitly, we can write $r_n(t) = \tanh(t/2)\xi_n$. Then
			\begin{align*}
				\xi_n \to \xi \text{ in the cone topology }&\iff r_n(\delta) \to r(\delta) \text{ for all } \delta > 0\\
				&\iff \|\tanh(\delta/2)\xi_n - \tanh(\delta/2)\xi\| \to 0 \text{ for all } \delta > 0\\
				&\iff \|\xi_n -\xi\| \to 0.
			\end{align*}
			We see that convergence in the cone topology on $X^\infty$ is precisely norm convergence
			for the induced Euclidean norm on $\mathbb{S}^{n-1}$. In particular, $X^\infty$ has a familiar and 
			rich topological structure with many compact subsets.
				\end{example}

\subsection*{The boundary cone $CX^{\infty}$}				
		Consider the product space $X^\infty\times \R_+$ endowed with the product of the cone topology
		on $X^\infty$ and the usual topology on $\R_+$. Define an equivalence relation $\sim$ on $X^\infty\times\R_+$
		by $(\xi,s) \sim (\xi',s')$ if $s = s' = 0$ or $(\xi,s) = (\xi',s')$.
		Now define the \textit{boundary cone} $CX^\infty$ as the quotient of $X^\infty \times\R_+$ by $\sim$, endowed with the 			quotient topology which we will also refer to as the cone topology (context should always prevent any confusion). 
		The equivalence
		class of $(\xi,s) \in X^\infty\times \R_+$ is denoted by $[\xi,s]$. We will sometimes use the notation $[0]$ to denote the 		equivalence class corresponding to $s = 0$. Note that continuity of a function defined on $CX^\infty$
		is equivalent to sequential continuity because of the sequentially characterized topology on $X^\infty \times \R_+$. 
		The cone topology on $CX^\infty$ has been used to study the geometry of 
		Wasserstein space $\mathscr{W}_2(X)$ where $X$ is a Hadamard
		space \cite{wassersteincone}.
As we have seen, when $X$ is the Euclidean space $\R^n$ or hyperbolic space $\mathbb{H}^n$, the boundary $X^{\infty}$ is homeomorphic to the Euclidean sphere $\mathbb{S}^{n-1}$ and the boundary cone $CX^{\infty}$ is homeomorphic to $\mathbb{R}^n$.  On the other hand, the boundary cone for the tripod is itself homeomorphic to the tripod.

		To shed some light on the cone topology for $CX^\infty$ we give a partial characterization
		of convergence in terms of convergence in $X^\infty \times \R_+$.  The proof is routine, but technical, so we defer it to the appendix.

		\begin{lemma}[\bf Convergence in the boundary cone]
			\label{lem:convergence}
			Let $\st{[\xi_n,s_n]}_{n=1}^\infty \seq CX^\infty$ be a sequence and let $[\xi,s]\in CX^\infty$. Then 
			\[ [\xi_n,s_n] \to [\xi,s] \begin{cases}
					\implies s_n \to s, & s = 0\\
					\iff s_n \to s \text{ and } \xi_n\to \xi, & s \neq 0.
			\end{cases}\]
			If $X$ is proper then the converse implication also holds in the case $s = 0$.
					\end{lemma}
\noindent
The converse implication in Lemma \ref{lem:convergence} can fail if $X$ is not proper:  see the appendix.

		Now define the \textit{pairing} on $X \times CX^\infty$ to be the function 
		\[\ip{\cdot,\cdot} \from X\times CX^\infty \to \R,~ ~ ~ \ip{x,[\xi,s]} = \begin{cases}
				sb_\xi(x), & s > 0\\
		0, & s = 0.\end{cases}\]
		This suggestive notation will be put to use in the next section, after we prove here some important properties 
		of the pairing.
		\begin{proposition}[\bf Properties of the pairing]
			\label{prop:paircont}
			The pairing $\ip{\cdot,\cdot}\from X\times CX^\infty \to \R$ has the following properties:
			\begin{enumerate}
				\item[(i)] It is continuous.
				\item[(ii)] For all $[\xi,s]\in CX^\infty$, the map $\ip{\cdot, [\xi,s]} \from X \to \R$ is geodesically convex and $s$-Lipschitz.
				\item[(iii)] The pairing is positively homogeneous in the second
					argument: $\ip{\cdot,[\xi,\alpha s]} =\alpha \ip{\cdot,[\xi,s]}$ for all $[\xi,s] \in CX^\infty$ and $\alpha \geq 0$.
				\end{enumerate}
			\end{proposition}
		\begin{proof}
			(i) Due to our earlier remarks on the sequential nature of the  
			topology of $CX^\infty$, it suffices to prove that the pairing is sequentially continuous.
			Suppose that $x_n \to x$ in $X$ and $[\xi_n,s_n]\to [\xi,s]$ in $CX^\infty$. If $s = 0$, then $s_n \to 0$
			by Lemma \ref{lem:convergence}, and furthermore:
			\[|b_{\xi_n}(x_n)| = |b_{\xi_n}(x_n) - b_{\xi_n}(\bar x)| \leq d(x_n,\bar x) \to d(x,\bar x).\]
			Thus the sequence $\st{b_{\xi_n}(x_n)}_{n=1}^\infty$ is a bounded sequence of real numbers, whence 
			\[\ip{x_n,[\xi_n,s_n]} = s_nb_{\xi_n}(x_n) \to 0 = \ip{x,[0]}\]
			since $s_n\to 0$. If $s > 0$ then $s_n\to s$ and $\xi_n\to \xi$
			by Lemma \ref{lem:convergence}, giving 
			\[|b_{\xi_n}(x_n) - b_{\xi_n}(x)| \leq d(x_n,x) \to 0.\]
			Knowing that $b_{\xi_n}(x) \to b_\xi(x)$ from the definition of convergence in $X^\infty$, 
			we deduce that $\st{b_{\xi_n}(x_n)}_{n=1}^\infty$ and $\st{b_{\xi_n}(x)}_{n=1}^\infty$ share the common limit $b_\xi(x)$.
			It follows that 
			\[\ip{x_n,[\xi_n,s_n]} = s_nb_{\xi_n}(x_n)\to sb_\xi(x) = \ip{x,[\xi,s]}.\]
			Thus $\ip{\cdot,\cdot}$ is continuous.

			(ii) Let $[\xi,s]\in CX^\infty$ and consider the map $x\mapsto \ip{x,[\xi,s]}$ from $X$ to $\R$. If $s = 0$ then this is the constant
			function zero which is trivially convex and $0$-Lipschitz. Otherwise, this map is $x\mapsto sb_\xi (x)$. Busemann functions
			are geodesically convex and 1-Lipschitz, and $s > 0$ so this map is also geodesically convex and 
			$s$-Lipschitz. 

			(iii) This property is immediate from the definition. 
			\end{proof}
			
Our choice of the cone topology for the boundary $X^{\infty}$ balances two benefits.  On the one hand, the topology is weak enough to allow a wealth of compact sets in $X^\infty$, and hence also in the boundary cone $CX^\infty$.  On the other hand, the topology is strong enough to allow a wealth of continuous functions on $CX^\infty$, namely $[\xi,s] \mapsto \ip{x,[\xi,s]}$ for $x\in X$.

		\end{section}

		\begin{section}{Busemann subgradients and envelopes}
		Our development relies fundamentally on the following definition.	
		\begin{definition}
			\label{def:epihoro}
			Consider a nonempty subset $C$ of $X$ and a real-valued function $f\from C \to \R$.
			A \textit{Busemann subgradient} of $f$ at a point $x\in C$ is an element $[\xi,s] \in CX^\infty$
			such that $x$ minimizes $y\mapsto f(y) - \ip{y,[\xi,s]}$ over $C$. The function $f$ is \textit{Busemann subdifferentiable} if it has a Busemann subgradient at every point in $C$.

	\end{definition}
	
\noindent
An immediate consequence of the definition is that a point $x \in C$ minimizes a function $f\from C\to\R$ if and only if $[0]$ is a Busemann subgradient of $f$ at $x$.
		
		 When the space $X$ is $\R^n$, the boundary at infinity can be identified with
		$\mathbb{S}^{n-1}$ and the Busemann function $b_\xi$ has the form
		$b_\xi(y) = (\bar x-y)^T\xi$ (recall the fixed reference point $\bar x$). Then Definition \ref{def:epihoro} says $[\xi,s]$ is a Busemann subgradient at $x$
		if and only if
		\[f(y) - s\xi^T(\bar x-y) \geq f(x) - s\xi^T(\bar x - x) \text{ for all } y\in C,\]
		where we interpret the expression $s\xi$ here and throughout as identically $0$ if $s = 0$.
		This is equivalent to
		\[f(y) \geq f(x) -s\xi^T(y-x) \text{ for all } y\in C.\]
		In other words, Busemann subgradients for functions on Euclidean space coincide with the usual notion of subgradient
		via the identification $[\xi,s] \leftrightarrow -s\xi$. 
		
		More generally, $[\xi,s]$ is a Busemann subgradient of $f\from C\to\R$ at $x$ if and only if
		\begin{equation}
			\label{eqn:busesubpre}
			f(y) - sb_\xi(y) \geq f(x) - sb_\xi(x) \text{ for all } y\in C.
		\end{equation}
		We adopt the same convention that $sb_\xi$ is identically $0$ if $s = 0$.
		The Busemann functions for asymptotic rays differ only by an additive constant 
		(Proposition \ref{additive}), so it follows that inequality
		(\ref{eqn:busesubpre}) is equivalent to 
		\begin{equation}
			\label{eqn:busesub}
			f(y) \geq f(x) + sb_{x,\xi}(y) \text{ for all } y\in C,
		\end{equation}
		a version of the subgradient inequality in Hadamard space.
		An immediate consequence of this inequality is
		that $f\from C\to \R$ is lower semicontinuous at any point where it has 
		a Busemann subgradient because Busemann functions themselves are continuous.
		We interpret Busemann subgradients $[\xi,s]$ for $f$ at $x$ as rays issuing from $x$ in direction
		$\xi$ with speed $s\geq 0$.
		
		The following example is central for us.  We revisit the result in a broader setting later in the section.

\begin{proposition}\textbf{(Busemann subgradients of distance functions)}
\label{distanceprop}
For any point $a\in X$, the function $x\mapsto d(x,a)$ is Busemann subdifferentiable
on $X$. At the point $a$, one Busemann subgradient is $[0]$.  At any point $x \ne a$, 
each ray $r$ issuing from $x$ and passing through $a$ gives rise to a Busemann subgradient $[r(\infty),1]$.
\end{proposition}

\begin{proof}
Since the point $a$ minimizes the distance function, $[0]$ is a subgradient.  On the other hand, at any point $x \ne a$, we must prove that all points $y \in X$ satisfy the inequality 
$d(y,a) \ge d(x,a) + b_r(y)$, where 
\[
b_r(y) = \lim_{t \to \infty} \{ d\big(y,r(t)\big) - t \}
\]
is the Busemann function for the ray $r$.  The argument on the right-hand side is nonincreasing in $t$, by \cite[Lemma II.8.18(1)]{bridson}, and equals $d(y,a) - d(x,a)$ when $t=d(x,a)$, so the result follows.
\end{proof}

	We next prove a basic chain rule for compositions of Busemann subdifferentiable
		functions with scalar convex functions. 

		\begin{proposition}\textbf{(Chain rule)}
			\label{prop:chain}
			Suppose $f\from C\to \R$ is Busemann subdifferentiable, and $g\from \R\to \R$ is a nondecreasing
			convex function. Then $g\circ f$ is Busemann subdifferentiable, and if $f$
			has a Busemann subgradient $[\xi,s]$ at $x \in C$ then for any $\alpha \in \p g(f(x))$, $[\xi,\alpha s]$ 
			is a Busemann subgradient for $g\circ f$ at $x$.
		\end{proposition}
		\begin{proof}
			Fix $x\in C$ and let $[\xi,s]$ be a Busemann subgradient for $f$ at $x$.
			For any $\alpha \in \p g(f(x))$ we must show $y\mapsto (g\circ f)(y) - \ip{y,[\xi,\alpha s]}$ is minimized
			over $C$ by $x$. Note that $\alpha \geq 0$ because $g$ is nondecreasing.
			Since $[\xi,s]$ is a Busemann subgradient for $f$ at $x$ we have
			\begin{equation}
				\label{eqn:subgradfunction}
				f(y) - \ip{y,[\xi,s]} \geq f(x) - \ip{x,[\xi,s]} \text{ for all } y\in C.
			\end{equation}
			Likewise, since $\alpha \in \p g(f(x))$ we have 
			\begin{equation}
				\label{eqn:subgradscalar}
				g(f(y)) - \alpha f(y) \geq g(f(x)) - \alpha f(x) \text{ for all } y\in C.
			\end{equation}
			Finally we estimate
			\begin{align*}
				g(f(x)) - \ip{x,[\xi,\alpha s]} &= g(f(x)) - \alpha \ip{x,[\xi,s]}\\
				&\leq \alpha (f(x) - \ip{x,[\xi,s]} - f(y)) + g(f(y))\\
				&\leq g(f(y)) - \alpha \ip{y,[\xi,s]} \\
				&= g(f(y)) - \ip{y,[\xi,\alpha s]}.
			\end{align*}
			The first and last equalities use positive homogeneity of the pairing in the second slot (Proposition \ref{prop:paircont}(3)), while
			the second line and third lines use (\ref{eqn:subgradscalar}) and (\ref{eqn:subgradfunction}) respectively.
			The inequality (\ref{eqn:subgradfunction}) is preserved because the coefficient $\alpha$ is nonnegative. 
		\end{proof}

		\begin{example}\textbf{(Reparametrized distance functions)}
			\label{ex:distp}
			Fix $a\in X$, $\sigma \geq 0, b\in \R$, and $p\geq 1$. Taking $f(x) = d(x,a)$ and $g(s) = \sigma s^p + b$, 
			we combine the result of Proposition \ref{distanceprop} with the chain rule (Proposition \ref{prop:chain})
			to immediately yield the Busemann subdifferentiability 
			of $g\circ f$ on $X$. More explicitly, $[0]$ remains a Busemann
			subgradient at $a$ and at each $x\in X \setminus
			\st{a}$	we obtain a Busemann subgradient $[r(\infty),p\sigma d(x,a)^{p-1}]$. 
		\end{example}
		
		Before we proceed to further examples, let us show that Busemann subdifferentiable functions defined on
		geodesically convex sets are
		geodesically convex.
		\begin{proposition}[\bf Subdifferentiability implies convexity]
			\label{prop:buseimpgeo}
			Suppose $C\seq X$ is geodesically convex and $f\from C \to \R$ is Busemann subdifferentiable. Then $f$ is 
			geodesically convex on $C$.
		\end{proposition}
		\begin{proof}
			Take any two points $x,y\in C$ and let $\lambda \in [0,1]$. 
			Let $\gamma \from [0,1] \to X$ parametrize the geodesic segment $[x,y]$, denote the point 
			$\gamma(\lambda)$ by $z_\lambda$, and note $z_\lambda \in C$ by geodesic convexity of $C$.
			Use the Busemann subdifferentiability of $f$ to procure a Busemann subgradient $[\xi,s]$
			at $z_\lambda$.
			Then two applications of Definition \ref{def:epihoro} yield:
			\[f(x) - \ip{x,[\xi,s]} \geq f(z_\lambda) - \ip{z_\lambda,[\xi,s]}\]
			\[\,~f(y) - \ip{y,[\xi,s]} \geq f(z_\lambda) - \ip{z_\lambda,[\xi,s]}.\]
			Multiply the first inequality by $\lambda$ and the second by $1-\lambda$, and sum the
			resulting inequalities to obtain
			\begin{equation}
			\label{eqn:convexity}
			\lambda f(x) + (1-\lambda)f(y) \geq f(z_\lambda) + \lambda \ip{x,[\xi,s]} + (1-\lambda)\ip{y,[\xi,s]} - \ip{z_\lambda,[\xi,s]}.
		\end{equation}
		Convexity of the pairing in the first argument (Proposition \ref{prop:paircont}(2)) implies
		\[\lambda \ip{x,[\xi,s]} + (1-\lambda)\ip{y,[\xi,s]}  \geq \ip{z_\lambda,[\xi,s]}.\]
			Then it follows from (\ref{eqn:convexity}) that 
			\[\lambda f(x) + (1-\lambda)f(y) \geq f(z_\lambda),\]
			so $f$ is geodesically convex on $C$ since $x,y\in C$ were arbitrary.
		\end{proof}
		
The following example shows that Busemann subdifferentiability essentially coincides with convexity in the Euclidean case.
		
		\begin{example}[\bf Euclidean Busemann subdifferentiability]
			\label{ex:euclid}
			Let $C\seq \R^n$ be convex and suppose $f\from C \to \R$ is convex. Classical convex-analytic arguments show that if $f$ is locally
			Lipschitz on $C$ then $f$ has a subgradient at each point in $C$, so $f$ is Busemann subdifferentiable. In light of Proposition \ref{prop:buseimpgeo}, we conclude that convexity and Busemann subdifferentiability coincide for locally Lipschitz functions on convex sets in Euclidean space.					\end{example}

More generally, however, the assumption of Busemann subdifferentiability, fundamental for the algorithms we develop, is more restrictive than geodesic convexity, as the following examples show.

		\begin{example}[\bf Subdifferentiability in the tripod]
			\label{ex:geonotbuse2}
			Consider the tripod $X$ and let $C \seq X$ be the union of any two of the three defining rays (including the origin).
			Then $C$ is geodesically convex so the distance function $f(x) = \dist(x,C)$ is geodesically convex, but we claim that $f$ 
			has no Busemann subgradient at any point outside of $C$.  To see this, suppose $x_0$ lies in the open ray disjoint from $C$. Assume $f$ has a Busemann subgradient $[\xi,s]$
			at $x_0$, with $s > 0$ because $x_0$ does not minimize $f$. Then we have the inequality
			\begin{equation}
				\label{eqn:tripodsub}
				f(y) - sb_\xi(y) \geq f(x_0) - sb_\xi(x_0) \text{ for all } y\in X.
			\end{equation}
		At least one of the rays defining $C$ must be distinct from the 
			ray emanating from the origin with direction $\xi$, so take $y$ to be a point in such a ray $r$. Then $f(y) = 0$,
			while $-sb_\xi(y)$ blows up to $-\infty$ as $y$ travels along $r$ away from the origin (see Example~\ref{ex:tripod}).
			For $y$ sufficiently far along $r$ this leads to a violation of (\ref{eqn:tripodsub}).
		\end{example}

The same difficulty can occur in Hadamard manifolds, as the next example shows.

		\begin{example}[\bf Subdifferentiability in the Poincar\'e disk]
			\label{ex:geonotbuse}
			We consider $X = \mathbb{H}^2$, 
			The subset $F = \st{(x,y) \in \mathbb{H}^2 \mid y = 0}$ is closed and geodesically convex 
			(it is a geodesic line in $\mathbb{H}^2$), so the function
			$f(p) = \dist(p,F)$ is geodesically convex (\cite[II.2 Corollary 2.5]{bridson}). 
			However, $f$ is not Busemann subdifferentiable on any neighborhood of $F$. To see this, let $U$ be a neighborhood of $F$, consider any $(x_0,y_0) \in U$ with $y_0\neq 0$, 
			and suppose there exists a 
			Busemann subgradient $[\xi,s]$ at $(x_0,y_0)$, with $s > 0$ because $(x_0,y_0)$ does not minimize $f$. 
			We can identify $\xi \in X^\infty$ with a point 
			$\xi = (\xi_1,\xi_2) \in \mathbb{S}^1$, and we will use this to write the Busemann
			function $b_\xi$ explicitly. By definition $f - \ip{\cdot,[\xi,s]}$ is minimized 
			at $(x_0,y_0)$, i.e.\
			\[\dist((t,0),F) - \ip{(t,0),[\xi,s]} \geq \dist((x_0,y_0),F) - \ip{(x_0,y_0),[\xi,s]} \text{ for all } t\in (-1,1).\]
			Since $(t,0) \in F$ for all $t\in (-1,1)$ this is equivalent to
			\begin{equation}
				\label{eqn:poincare2}
				-sb_{\xi}(t,0) \geq \dist((x_0,y_0), F) - \ip{(x_0,y_0),[\xi,s]} \text{ for all } t\in (-1,1).
		\end{equation}
			Choosing our reference point to be $(\bar x,\bar y) = (0,0)$, \cite[II.8.24(2)]{bridson} 
			tells us the Busemann function $b_\xi$ has the following
			form:
			\begin{equation}
				\label{eqn:busepoin}
				b_\xi( x,y ) = - \log \left(\frac{ 1 - x^2 - y^2 }{(x-\xi_1)^2 + (y-\xi_2)^2}\right).
		\end{equation}
			Substituting into (\ref{eqn:poincare2}) gives
			\[s \log\left(\frac{1-t^2}{(t-\xi_1)^2 + \xi_2^2}\right) \geq \dist ( ( x_0,y_0),F) - \ip{(x_0,y_0),[\xi,s]}.\]
			If $\xi_2 \neq 0$ let $t\to 1$, and if $\xi_2 = 0$ let $t\to -\xi_1$. In either case
			the lefthand side approaches $-\infty$ while the righthand side remains constant, a contradiction.
		\end{example}
		
Notwithstanding these examples, we next argue that the class of Busemann subdifferentiable functions is quite rich.

\subsection*{Busemann envelopes}
Given any function $g\from CX^\infty \to (-\infty,+\infty]$, we can define a {\em conjugate} function 
$g^\circ \colon X \to [-\infty,+\infty]$ by
\begin{equation} \label{preconjugate}
g^\circ(x) ~=~ \sup\st{\ip{x,[\xi,s]} - g([\xi,s]) \from [\xi,s] \in CX^\infty} \qquad (x \in X).
\end{equation}
Conjugacy gives a fundamental way to construct nontrivial Busemann subdifferentiable functions on a general Hadamard space.

\begin{definition}
A function $f \from X \to \R$ is a {\em Busemann envelope} if there exists a function 
$g\from CX^\infty \to (-\infty,+\infty]$ such that $f = g^\circ$ with the supremum (\ref{preconjugate}) attained for all points $x \in X$.
\end{definition}
 
\begin{proposition}  \label{ex:busemax}
Busemann envelopes are Busemann subdifferentiable.
\end{proposition}

\begin{proof}
Any $[\xi,s]$ attaining the supremum in (\ref{preconjugate}) for a given point $x\in X$ gives
\[
g^\circ(y) -\ip{y,[\xi,s]} \geq - g([\xi,s]) = g^\circ(x) - \ip{x,[\xi,s]}  \text{ for all } y\in X.
\]
This says $[\xi,s]$ is a Busemann subgradient at $x$ according to Definition \ref{def:epihoro}.
\end{proof}

			\begin{example}\textbf{(Euclidean convex functions)}
				\label{ex:convcont}
On Euclidean space, all finite convex functions are Busemann envelopes.  The proof involves routine classical convex analysis.  While reassuring, our development does not depend on this observation, so we defer the argument to the appendix.
			\end{example}

			\begin{example}\textbf{(Busemann functions)}
	\label{ex:buseex}
	The most basic class of Busemann envelopes are Busemann functions themselves.
	Suppose $r$ is a ray on $X$.  Then the corresponding Busemann function $b_r$ has Busemann
	subgradient $[r(\infty),1]$ at each point $x\in X$.  Furthermore, for the function 
	$g\from CX^\infty \to (-\infty,+\infty]$ taking the value $0$ at $[r(\infty),1]$ and $+\infty$ elsewhere, we have $b_r = g^\circ$, and we easily check that $b_r$ is therefore a Busemann envelope.
	\end{example}

A much larger class of Busemann envelopes is furnished by the following example.

	\begin{example}
				\textbf{(Continuous-compact-representable functions)}
				\label{ex:sublinear}
				Endowing the boundary $X^\infty$ with the cone topology opens up the following natural construction.
				Suppose $K \seq CX^\infty$ is compact and $g\from CX^\infty \to (-\infty,+\infty]$ is
				continuous on $K$ and
				$+\infty$ outside of $K$. Because the pairing is also continuous (Proposition \ref{prop:paircont}(1)), the function $g^\circ$ is a Busemann envelope,  and hence Busemann subdifferentiable by Proposition \ref{ex:busemax}.  This construction generates a rich class of examples.  In particular, in Euclidean space we recover the class of ``almost sublinear'' functions:  those functions $f$ such that $\sup_{x\in \R^n}|f(x)-h(x)|< \infty$ for some sublinear function $h$ \cite[Proposition 4.5]{borweinalmost}. This class covers all convex optimization problems with compact feasible regions, by extending the objectives to have this growth property outside the feasible region.  More generally, consider the case when the space $X$ is proper, and hence the boundary $X^\infty$ is compact.  Then compact subsets of the boundary cone $CX^\infty$ are abundant, because if $F\seq X^\infty$ is closed and $I \seq \R_+$ is compact, then the image of $F\times I$ under the quotient map is compact.
					\end{example}

We next observe that we we can view distance functions as Busemann envelopes.
					
						\begin{example}\textbf{(Distance functions)}
			\label{ex:dist}
			Consider the function $d(\cdot,a)$ for any fixed $a\in X$.
			For each $x\neq a$ let $r_{x,a}\from \R_+ \to X$
			be any ray issuing from $x$ that eventually passes through $a$ (at least one exists by the 
			geodesic extension property). Furthermore, define $\xi_x = r_{x,a}(\infty)$, 
			and let $K = \st{[\xi_x,1] \mid x \in X \setminus \st{a}}$.
			Define $g\from CX^\infty \to (-\infty,+\infty]$ by
			\[
				g([\xi,s]) = \begin{cases}
					\ip{a,[\xi,s]}, & [\xi,s] \in K \\
					+\infty, & [\xi,s] \notin K.
			\end{cases} \]
Then it is straightforward to check $d(\cdot,a) = g^\circ$, and furthermore that $d(\cdot,a)$ is consequently a Busemann envelope.  Since the details are routine, and we have already proved Proposition \ref{distanceprop}, we defer them to the appendix.			
							\end{example}
			
					\begin{remark}[\bf Finite max functions]
			\label{rem:max}
			If $f_1,\dots, f_m$ are Busemann subdifferentiable functions 
			on a common set $C$, then $\max\st{f_1,\dots,f_m}$ is Busemann subdifferentiable on $C$: 
			a Busemann subgradient at $x\in C$ is obtained by
			choosing a Busemann subgradient for any function attaining the maximum.
		\end{remark}

		\begin{example}\textbf{(Distance to balls and horoballs)}
			\label{ex:distballs}
			Generalizing Proposition \ref{distanceprop}, we can show that the distance functions
			to balls and horoballs are Busemann subdifferentiable.
		Fix any $a\in X, \rho \geq 0$ and consider the function $f(x) = \dist(x,B_\rho(a))$
		where $B_\rho(a)$ is the closed ball of
		radius $\rho$ around $a$. Then $f(x) = \max\st{0,d(x,a)-\rho}$ and the Busemann subdifferentiability of $f$ follows from Example
		\ref{ex:distp} and Remark~\ref{rem:max}.
Turning to the case of horoballs, consider a ray $r$ in $X$ and the corresponding horoball $H_r = b_r^{-1}( (-\infty,0])$. We then have
\[
\dist(x,H_r) = \max\st{0,b_r(x)}, 
\]
proving Busemann subdifferentiability by Example 
		\ref{ex:buseex}
		and Remark \ref{rem:max}.  For details, see the appendix.
			\end{example}

\subsection*{Fenchel conjugacy}
We next detour to compare our approach and the conjugate construction (\ref{preconjugate}) with related usages of Busemann functions for convex analysis in Hadamard spaces. 
		Given a geodesically convex function $f\from X\to \R$, the recent work
		\cite{convanalyshad} defines the \textit{asymptotic Legendre-Fenchel conjugate} $f^*\from CX^\infty \to (-\infty,+\infty]$ 
		by
		\[f^{*}([\xi,s]) = \sup_{x\in X}\st{-\ip{x,[\xi,s]} - f(x)}.\]
		Note that this definition depends implicitly on the choice of reference point $\bar x$, but only up
		to an additive term (Proposition \ref{additive}). That is, if we use subscripts to denote the choice of reference point
		then
		\[f^*_{\bar {x}}([\xi,s]) = f^*_{\hat x}([\xi,s]) -\ip{\hat{x},[\xi,s]} \text{ for all } [\xi,s]\in CX^\infty.\]
		The definition of $f^*$ yields a version of the classical Fenchel-Young inequality:
		\[f^*([\xi,s]) + f(x) \geq -\ip{x,[\xi,s]}  \text{ for all } [\xi,s] \in CX^\infty, x\in X.\]
		If $f \from X\to \R$ is Busemann subdifferentiable, any Busemann subgradient $[\xi,s]$ at $x\in X$ gives the 
		following inequality by definition:
		\[f(y) - \ip{y,[\xi,s]} \geq f(x) - \ip{x,[\xi,s]} \text{ for all } y\in X.\]
		Rearranging and taking the supremum over $y\in X$ leads to a different type of conjugate $f^\bullet \from CX^\infty \to (-\infty,+\infty]$:
		\begin{equation}
			\label{eqn:prefench}
			\ip{x,[\xi,s]} - f(x) \geq \sup_{y\in X}\st{\ip{y,[\xi,s]} - f(y)} =: f^\bullet([\xi,s]).
	\end{equation}
By definition, we have the Fenchel-Young-type inequality
		\begin{equation}
			\label{eqn:prefench2}
			f^\bullet([\xi,s]) + f(x) \geq \ip{x,[\xi,s]} \text{ for all } [\xi,s]\in CX^\infty, x\in X,
	\end{equation}
which holds with equality if and only if $[\xi,s]$ is a Busemann subgradient for $f$ at $x$, analogously to the classical case. The only difference between the 
		definitions of $f^*$ and $f^\bullet$ is the sign of the pairing term. 
		For comparison, \cite[Definition 1]{fenchelhadam} concerns a notion of Fenchel conjugate on Hadamard manifolds with the sign on the pairing term matching inequality \eqref{eqn:prefench}.
		
Taking one step further along the classical path, we might consider the biconjugate $(f^\bullet)^\circ$.  In linear spaces, linking convexity to the identity $f = f^{**}$ is the famous result of Fenchel-Moreau.
A similar result holds in Hadamard spaces, using Busemann subdifferentiability in place of convexity.   While interesting formally, our development does not rely on this result, so we defer its proof to the appendix.
	
\begin{proposition}[\bf Fenchel biconjugation]
\label{biconjugate}
Consider a function $f\from X\to \R$. If $f$ is Busemann subdifferentiable then 
$f= (f^\bullet)^\circ$.  If $X$ is proper and $f$ is continuous, the converse also holds.
\end{proposition}

\section{The geometry of Busemann subdifferentiability}
Classical convex analysis profits from the relationship between the convexity of functions, the convexity of their level sets and epigraphs, and their Lipschitz continuity.  In this section, we explore some parallels in the Hadamard space setting, in which Busemann subdifferentiability of functions is related to the notion of horospherical convexity of sets, a geometric property defined in terms
			of Busemann functions.
		\begin{definition}
			\label{def:horoconv}
		For a closed set $F \seq X$ we say $F$ is \textit{horospherically convex} if for each 
		point $x$ in the boundary of $F$, there
		exists a \textit{supporting ray at} $x$, which is to say a ray $r$ issuing from $x$ such that 
		\[F\seq \st{z\in X\mid b_r(z)\leq 0}.\]
		Then we call the righthand side a \textit{supporting horoball} for $F$ at $x$, and say 
		the ray $r$ \textit{supports} $F$ at $x$.  
	\end{definition}

\noindent	
In Euclidean space, horoballs are halfspaces:  each supporting horoball for a convex set corresponds to a supporting hyperplane, and thus to a ``normal'' vector, a local geometric idea based on angle.  This relationship breaks down in nonlinear spaces: we make no use of normality, relying instead on supporting horoballs and rays.  
	
	The Busemann functions appearing in both Definitions \ref{def:epihoro} and \ref{def:horoconv}
	hint at the following relationship between Busemann subdifferentiability and horospherically convex level sets, mimicking the classical case.
		\begin{proposition}[\bf Subdifferentiability and level sets]
			\label{prop:levhoro}
			Consider a Busemann subdifferentiable function $f\from X\to \R$.  If either $f$ is continuous or the space $X$ is proper, then for any value $M > \inf_{X} f$, the level set
			$\st{x\in X \mid f(x) \leq M}$ is horospherically convex.
					\end{proposition}
		\begin{proof}
Denote the level set by $F$.  We consider a point $\bar x$ is in the boundary of $F$. 
			
Suppose first that the function $f$ is continuous.  In that case we have $f(\bar x) = M$. Since $f$ is Busemann subdifferentiable there exists a Busemann subgradient $[\xi,s]$ at $\bar x$, so let $r$ be the ray issuing from $\bar x$ with
			 direction $\xi$.
			We will show that $r$ supports $F$ at $\bar x$. Since $M > \inf_X f$ we have
			$\bar x\notin \argmin_X f$, implying $s > 0$. Then (\ref{eqn:busesub}) implies 
			\[
				s b_{\bar x,\xi}(z) + f(\bar x)  \leq f(z)  \text{ for all } z\in X.
		\]
		If $z\in F$ notice $f(\bar x) = M \geq f(z)$ so in fact $s b_{\bar x,\xi}(z) \leq 0$. Dividing by $s > 0$ 
		implies $b_{\bar x,\xi}(z)\leq 0$. The equality $b_{\bar x,\xi} = b_r$ 
		is only a matter of notation,
		so in this way we obtain a supporting horoball at every boundary point of $F$, rendering $F$ horospherically convex.
	
Now suppose instead that the space $X$ is proper. Some sequence
		$x_k\to \bar x$ satisfies $f(x_k) > M$. At each $x_k$, choose a Busemann subgradient $[\xi_k,s_k]$. Then for all $y\in F$,
		\[M + s_kb_{x_k,\xi_k} < f(x_k) + s_k b_{x_k,\xi_k}(y) \leq f(y) \leq M.\]
		Note $s_k > 0$ because $x_k$ does not minimize $f$, hence $b_{x_k,\xi_k}(y)\leq 0$ for all $y\in F$. Properness of $X$ implies that
		$X^\infty$ is compact, allowing us to extract a convergent subsequence $\xi_k \to \xi$. The map $(x,\xi)\mapsto
		b_{x,\xi}(y)$ is continuous \cite[II.1]{Ballmann1995Lectures} so passing to the limit implies
		$b_{x,\xi}(y) \leq 0$ for $y\in F$. Thus $F$ admits a supporting horoball at each boundary point
		as desired.
\end{proof}

On proper spaces $X$, as we shall see, Busemann subdifferentiable functions $f$ are in fact always continuous.  Nonetheless, the use of continuity above is illuminating for what follows.  It is also instructive now to return to Examples \ref{ex:geonotbuse2} and \ref{ex:geonotbuse}.  These examples concern continuous functions on proper Hadamard spaces, but one may verify that their level sets are not horospherically convex.  Proposition \ref{prop:levhoro} therefore implies that the functions are not Busemann subdifferentiable, as we verified directly.

	Continuing the classical analogy, the next result shows Busemann subdifferentiability of a function $f$
	is closely related both to horospherical
	convexity of its \textit{epigraph}, 
	\[
		\epi f := \st{(x,\alpha) \in X\times \R \mid f(x) \leq \alpha}, 
	\]
and to continuity.
While interesting, we will not rely on the following result, and since its proof resembles that of Proposition \ref{prop:levhoro}, we defer it to the appendix.

	\begin{proposition}[\bf Subdifferentiability, epigraphs and continuity]
		\label{prop:horoepi} \mbox{} \\
		For a function \mbox{$f\from X\to \R$}, consider the following three conditions:
\begin{enumerate}[label=(\roman*)]
\item
$f$ is Busemann subdifferentiable
\item
the epigraph of $f$ is horospherically convex
\item
$f$ is Busemann subdifferentiable and continuous.

\end{enumerate}
The implications (iii) $\Rightarrow$ (ii) $\Rightarrow$ (i) hold, and if the space $X$ is proper, 
then all three properties are equivalent.
\end{proposition}

Just as in Euclidean space, continuous convex functions on the space $X$ must be locally Lipschitz, by \cite[Lemma 4.3]{sumrule}.  The following result relates Busemann subgradients to the local Lipschitz constant:  for an analogous Euclidean result, see \cite[Theorem 3.61]{beckfirstorder}.
		
		\begin{proposition}[\bf Busemann subdifferentiability and Lipschitz functions]
			\label{prop:liphoro}
			Consider a function $f\from C\seq X \to \R$ that is Busemann subdifferentiable. 
			\begin{enumerate}[label=(\roman*)]
				\item If at each point in $C$ the function $f$ admits a Busemann subgradient $[\xi,s]$ with 
					$s \leq L$, then $f$ is $L$-Lipschitz
					on $C$.

				\item If $f$ is $L$-Lipschitz on $C$ and $C$ is open, then every Busemann subgradient $[\xi,s]$
					for $f$ at $x$ in $C$ satisfies 
					$s \leq L$.
			\end{enumerate}
					\end{proposition}
		\begin{proof}
			(i) For $x,y\in C$ choose the assumed Busemann subgradients $[\xi_x,s_x],[\xi_y,s_y]$ at $x,y$ respectively.
				Applying Definition \ref{def:epihoro} twice gives:
				\begin{equation}
					\label{eqn:lip}
			 \begin{aligned}
				 f(x) - \ip{x,[\xi_x,s_x]} &\leq f(y) - \ip{y,[\xi_x,s_x]},\\
				 f(y) - \ip{y,[\xi_y,s_y]} &\leq f(x) - \ip{x,[\xi_y,s_y]}.
		\end{aligned}
		\end{equation}
		Rearranging the first inequality in (\ref{eqn:lip}) implies
		\[f(x) - f(y) \leq \ip{y,[\xi_x,s_x]} - \ip{x,[\xi_x,s_x]} \leq s_xd(x,y) \leq Ld(x,y),\]
		where the second inequality comes from Proposition \ref{prop:paircont}(2). Arguing similarly for the second inequality
		in (\ref{eqn:lip}), we find $f$ is $L$-Lipschitz on $C$.

			(ii) Suppose $f$ is $L$-Lipschitz on $C$ and take any Busemann subgradient $[\xi,s]$ at $x\in C$.
			If $s = 0$ there is nothing to prove, so we assume $s > 0$.
			Let $r$ be the ray issuing from $x$ with direction $\xi$.
			Since $X$ has the geodesic extension property
			$r$ can be extended to a 
			geodesic line $\tilde{r} \from \R \to X$ by \cite[Lemma II.5.8(2)]{bridson}.
			The geodesic line $\tilde{r}$ is continuous with $\tilde{r}(0) = x\in C$, and $C$ is open so
			there exists $\eps > 0$ such that $\tilde{r}(-\eps) \in C$.
			By (\ref{eqn:busesub}) we have
			\[s b_{x,\xi}(y) + f(x)  \leq f(y) \text{ for all } y\in X.\]
			Plugging in $y_\eps = \tilde{r}(-\eps)$ and using the definition of $b_{x,\xi} = b_r$ we find
			$b_{x,\xi}(y_\eps) = \eps$ from which we derive:
			\[s\eps = s b_{x,\xi}(y_\eps) \leq f(y_\eps) - f(x) \leq 
			Ld(y_\eps,x) =  L\eps.\]
			Divide through by $\eps > 0$ to conclude $s \leq L$.
		\end{proof}

			\begin{remark}
				\label{rem:distnotbuse}
				Proposition \ref{distanceprop} showed that $x\mapsto \dist(x,B)$ 
				is Busemann subdifferentiable when $B$ is a singleton, and Example \ref{ex:distballs}
				generalized this conclusion to sets $B$ that are balls or horoballs.
				Given the relationship between Busemann subdifferentiability and horospherical convexity discussed in Proposition \ref{prop:levhoro},
				one might ask if the distance function to a horospherically
				convex set is Busemann subdifferentiable. However, this result fails without more structure:  even in Euclidean space, horospherically convex sets may not even be convex (an example being a set of two distinct points).  				
				 \end{remark}

				 The next example shows that the behaviour illustrated in Remark \ref{rem:distnotbuse}
				 persists even for distances to sets that are both horospherically and geodesically convex.

				 \begin{example}\textbf{(Failure of subdifferentiability for distance functions)}
				\label{ex:distnotbuse2}
				A simple but interesting example of a Hadamard space $(X,d)$ can be obtained by
				gluing five Euclidean quadrants along 
		their edges to form a cycle, which can be realized concretely as the union of the following quadrants in $\R^3$
		endowed with the intrinsic metric induced by Euclidean distance:
		\[\R_+\times \R_+ \times\st{0}, ~
			\R_+\times\R_-\times\st{0}, ~ \R_-\times\R_+\times\st{0}, ~ \st{0}\times \R_-\times\R_+, ~ \R_-\times\st{0}\times\R_+.
	\]
	This space arises as a subspace of the tree space $\mathcal{T}_4$ (see Section \ref{sec:comp}), and is illustrated in 
	\cite[Figure 12]{billeratree} as well as Figure \ref{fig:levset1}. We will revisit this space when discussing the Busemann subdifferentiability
	of a sum of Busemann subdifferentiable functions in Example \ref{ex:nonhoro}.

	Define $C \seq X$ to be the geodesic segment joining the points $(1,0,0)$ and $(0,1,0)$. It is not hard to see that $C$ is 
	horospherically convex by considering rays parallel to $r(t) = (t/\sqrt{2})(1,1,0)$, and $C$ is obviously geodesically convex. 
	But one can show that if $f(x) = \dist(x,C)$
	then $f$ is not Busemann subdifferentiable at $x_0 = (1,1,0)$. The argument is structurally similar to 
	those of Examples \ref{ex:geonotbuse2}, \ref{ex:geonotbuse}, but requires checking a few cases so
	we relegate this casework to the appendix. 
			 \end{example}
			 
\subsection*{Subgradient-like algorithms}
The first subgradient-like algorithm for geodesically convex optimization in general Hadamard spaces appeared in \cite{lewis2024horoballs}.  That ``horospherical subgradient method'' relies only on the geometry of the objective level sets, specifically their horospherical convexity:  steps follow supporting rays to level sets.  That geometric framework is restrictive.  In particular, sums of such ``horospherically quasiconvex'' functions may not be horospherically quasiconvex, even in Euclidean space.

To circumvent this obstacle, we focus on structured objectives, specifically those that decompose as sums, aiming to handle each summand separately, in a stochastic or splitting fashion.  However, any such splitting method must rely on more than simply the level set geometry of the summands:  for example, $p$-mean problems involve summands whose level-set geometry is  independent of the exponent \mbox{$p \ge 1$,} even though the solutions vary with $p$.  In contrast with the supporting ray oracle used in the horospherical subgradient method, Busemann subgradients are comprised of both a ray and a ``speed'' at which that ray is traversed.  This extra information is crucial for our approach.

While Busemann subgradients are more informative than supporting rays to level sets, we must nonetheless take a splitting approach to structured optimization.  The reason is simple, but striking:  unlike classical convex analysis, {\em addition does not preserve Busemann subdifferentiability}.  We outline a simple counterexample below, with the details in the appendix.
	
\begin{example}\textbf{(Non-Busemann subdifferentiability of a sum)}
	\label{ex:nonhoro}
		Recall the quadrant space $X$ of Example \ref{ex:distnotbuse2}.
	Let $a_1 = (0,-1,0), a_2 = (-1,0,0)$ be points in $X$ and define $f\from X\to \R$ by
	\[f(x) = \frac{1}{2}d(x,a_1)^2 + \frac{1}{2} d(x,a_2)^2. \]
	By Example \ref{ex:distp}, $f$ is a sum of Busemann subdifferentiable functions.
	Taking $\bar x = (1/4, 1/4,0)$, we claim that the level set $f_{\bar x} = \st{z \in X \mid f(z) \leq f(\bar x)}$ is not horospherically convex. In the appendix, we show that there is no supporting horoball for $f_{\bar x}$ at $\bar x$. 		
	This example is illustrated in Figure \ref{fig:levset1}, where the black point is $\bar x$,
	the violet set corresponds to the horoball at $\bar x$ generated by the ray moving towards the
	\textit{spine} (the half-line $x = 0, y = 0, z\geq 0$), 
	and the red set is the level set $f_{\bar x}$. Only the parts of the horoball and level set contained in the 
	plane $z = 0$ are shown. Since $f$ is continuous, the proof of Proposition \ref{prop:levhoro} implies that it cannot be Busemann subdifferentiable:  otherwise, this level set would be horospherically 
		convex. In particular, $f$ is not Busemann subdifferentiable at $\bar x$.
		\begin{figure}
						\centering
			\includegraphics[scale=0.55]{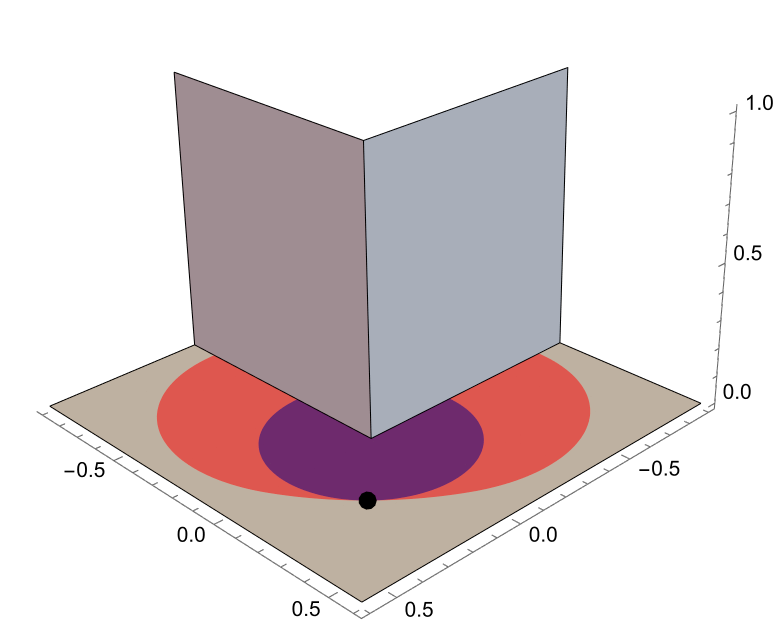}
			\caption{Part of a level set that fails to be horospherically convex.}
			\label{fig:levset1}
		\end{figure}
	\end{example}

\section{Subgradients in Alexandrov space}
Despite its generality, our development in this work is limited to the nonpositively-curved setting of Hadamard spaces.  Loosening this significant restriction to allow curvature that is bounded above but possibly positive gives rise to a more general class:  Alexandrov spaces. Convexity of distance functions then only holds on domains of small diameter.  (Euclidean spheres, for example, require restricting to small spherical caps.)
	Nonetheless, on such domains, $p$-mean problems (for example) can be solved
	via Riemannian gradient descent, on manifolds \cite{Afsari2012}, and more generally
	via proximal methods \cite{OhtaPalfia2015}. However, beyond Hadamard spaces, Busemann subgradients seem unhelpful:  geodesic rays do not exist in spheres, for example.  Whether some other notion of subgradient might support optimization algorithms on Alexandrov spaces with a satisfactory complexity analysis remains an open question.  As background, we here compare 
	Busemann subgradients with a recent notion of subgradient for convex functions on Alexandrov spaces 
\cite{sumrule}. 
For simplicity, we restrict attention to our usual setting of a Hadamard space $X$.  We briefly review some necessary tangent space constructions:  for details, see \cite{bridson}.  For any point $x \in X$, we denote by $\Theta_x X$  the set of all nonconstant geodesics issuing at~$x$.

		Given three points $x, y, z \in X$, a \textit{geodesic triangle} $\Delta = \Delta(x, y, z)$
	is the union of three geodesic segments (its sides) joining each pair of points. A
	\textit{comparison triangle} for $\Delta$ is a triangle  $\Delta(\bar x, \bar y,  \bar z)$ in $\R^2$
	that has side lengths equal
	to those of $\Delta$. 
	Take two geodesics $\gamma,\eta \in \Theta_x$, let $\gamma_t = \gamma(t), \eta_s = \eta(s)$ for $t,s\geq 0$ and 
	let $\Delta(\bar \gamma_t,\bar x,\bar \eta_s)$ be a comparison triangle in $\R^2$ for $\Delta(\gamma_t,x,\eta_s)$. Then the
	angle $\angle \bar \gamma_t\bar x\bar \eta_s$ is a nondecreasing function of both $t$ and $s$ and the
	\textit{Alexandrov angle} between $\gamma$ and $\eta$ is well-defined as the following limit: 
	\[\angle(\gamma,\eta) := \lim_{t,s\downarrow 0} \angle \bar \gamma_t \bar x\bar \eta_s.\]
		The Alexandrov angle induces a metric on the set $\Sigma_x X$ of equivalence classes of geodesics in $\Theta_x X$, where two geodesics $\gamma, \eta\in \Theta_x X$ are considered equivalent if $\angle(\gamma,\eta) = 0$. 
For a geodesic $\gamma \in \Theta_x X$, we denote by $[\gamma]$ its equivalence class. 
 The {\it tangent space} $T_x X$ of $X$ at $x$ is the Euclidean cone over the metric space $(\Sigma_x X, \angle)$ (see \cite[Chapter I, Definition 5.6]{bridson}). 
 For $v = ([\gamma],r), w = ([\eta],s) \in T_x X$ the {\it scalar product} of $v$ and $w$ is given by $\langle \langle v,w \rangle 
 \rangle = rs\cos \angle(\gamma, \eta)$. For a geodesically convex function $f\from X\to \R$, the \textit{subdifferential} of $f$ at $x$, denoted by
 $\p f(x)$, can be characterized (see \cite[Proposition 4.4]{sumrule})
 as the elements $([\eta],s)\in T_xX$ satisfying the \textit{subgradient inequality}
 \begin{equation}
	 \label{eqn:cat0subgrad}
 \ip{\langle ([\eta],s),([\gamma_y],d(x,y))\rangle} + f(x) \leq f(y) \text{ for all } y\in X,
 \end{equation}
 where $\gamma_y\from[0,d(x,y)]\to X$ denotes the geodesic from $x$ to $y$. Any element $v\in \p f(x)$ is called a \textit{subgradient}
 of $f$ at $x$.  A related idea used in the study of gradient flows in Alexandrov spaces is the \textit{minus-subdifferential} \cite[Definition 3.6]{GigliNobili2021}. 
 
 Busemann subgradients are a special case of the Alexandrov-space subgradient described above.  Specifically, we have the following result.   
 
		\begin{theorem}
			\label{thm:subgrad}
Given a Hadamard space $X$ and $C \seq X$, suppose that the function $f\from C\to \R$ is geodesically convex, with a Busemann subgradient $[\xi,s]$ at $x\in C$. 
			Let $r$ be a ray issuing from $x$ with direction $\xi$.
			Extend $r$ to a geodesic line
			$\tilde{r}\from \R\to X$ and define a new ray $r_-\from \R_+\to X$ by $r_-(t) = \tilde{r}(-t)$.
			Then $([r_-],s) \in \p f(x)$.	
		\end{theorem}
		\begin{proof}
			By definition of the scalar product on $T_xX$, the subgradient inequality we wish to prove is
			\[s d(x,y)\cos\angle(r_-,\gamma_y) + f(x)\leq f(y) \text{ for all } y\in X.\]
			Busemann subdifferentiability of $f$ implies, by (\ref{eqn:busesub}), 
			\[s b_r(y) + f(x) \leq f(y) \text{ for all } y\in X.\]
			Thus it is enough to prove (the case $s = 0$ being trivial):
			\begin{equation}
				\label{eqn:desired}
				b_r(y) \geq d(x,y)\cos\angle(r_-,\gamma_y) \text{ for all } y\in X.
			\end{equation}
			We proceed in two steps, first proving 
			$\cos\angle(r_-,\gamma_y) \leq -\cos\angle(r,\gamma_y)$ $(*)$.
			Using the triangle inequality for angles we find:
			\[\pi = \angle(r,r_-) \leq \angle(r_-,\gamma_y) + \angle(r,\gamma_y).\]
			Hence $\pi - \angle(r,\gamma_y)\leq \angle(r_-,\gamma_y)$, and taking the cosine reverses the inequality: 
\begin{equation} \label{star}
\cos\angle(r_-,\gamma_y) \leq \cos(\pi - \angle(r,\gamma_y)) = -\cos\angle(r,\gamma_y).
\end{equation}

			We next prove 
\begin{equation} \label{starstar}
b_r(y) \geq -d(x,y)\cos\angle(r,\gamma_y).
\end{equation}
			By the law of cosines we have:
			\[d(r(t),y)^2 \geq d(x,r(t))^2 + d(x,y)^2 - 2d(x,r(t))d(x,y)\cos\angle(r,\gamma_y).\]
			Rearrangement leads to
			\begin{equation}
				\label{eqn:ineqprelim}
				\frac{d(r(t),y)^2 - t^2}{2t} \geq \frac{d(x,y)^2}{2t} - d(x,y)\cos\angle(r,\gamma_y).
		\end{equation}
		The lefthand side of \eqref{eqn:ineqprelim} simplifies in the limit because $t\mapsto d(r(t),y)-t$ is bounded:
			\begin{align}
				\label{eqn:quadraticbuse}
			\lim_{t\to \infty}\frac{d(r(t),y)^2 - t^2}{2t} &= 
			\lim_{t\to \infty}(d(r(t),y)-t)\lim_{t\to \infty}\frac{d(r(t),y) +t}{2t} \nonumber\\
			&= b_r(y)\lim_{t\to \infty}\left(\frac{d(r(t),y)-t}{2t} + 1 \right)  \nonumber \\
			&= b_r(y).
	\end{align}
	Using \eqref{eqn:quadraticbuse} and sending $t\to \infty$ in (\ref{eqn:ineqprelim}) yields 
	(\ref{starstar}).
			Combining (\ref{star}) and (\ref{starstar}) gives
			\[d(x,y)\cos\angle(r_-,\gamma_y) \leq -d(x,y)\cos\angle(r,\gamma_y) \leq b_r(y).\]
			This proves (\ref{eqn:desired}).
		\end{proof}
		
\noindent
In fact, the proof above shows that Busemann subgradients induce elements of the minus-subdifferential.
		
In general, the idea of a Busemann subgradient is more restrictive than the Alexandrov space notion.  Indeed, on Alexandrov spaces, continuous geodesically convex functions  have subgradients everywhere
\cite[Theorem 4.11]{sumrule}, but in the special case of Hadamard spaces, we have seen (Example \ref{ex:geonotbuse}) that such functions may not be Busemann subdifferentiable.  

As a concrete example, consider Example \ref{ex:nonhoro}.  The function $f$ has no Busemann subgradient at the point $\bar x$, but it does have a subgradient there in the Alexandrov space sense.
	Consider the geodesic $\eta \from [0,1]\to X$ given by
	$\eta(t) = \bar x + (t/\sqrt{2})(1,1,0)$ in the quadrant $\R_+\times \R_+ \times \st{0} \seq X$.
	We show in the appendix that $([\eta],3/\sqrt{2}) \in T_{\bar x}X$ is a subgradient for $f$ at $\bar x$. Roughly speaking, the fact that Busemann subgradients correspond to subgradients but not conversely is due to a local/global divide:  subgradients are characterized by local geometry, whereas the Busemann subgradient property is a global notion.
			\end{section}

	\begin{section}{Subgradient-based minimization of sums}
		\label{sec:algs}
We now return to our original optimization problem in Hadamard space.  We seek stochastic and incremental subgradient-type methods for minimizing a function of the form
\[
f(x) ~=~ \sum_{i=1}^m f_i(x) \qquad (x \in C).
\]	
We collect our assumptions below. 

We denote the projection
			onto a nonempty closed geodesically convex set $C$ by $P_C$ 
			(well-defined by \cite[Proposition II.2.4]{bridson}).	
			Define an oracle \texttt{Busemann}
			that accepts a Busemann subdifferentiable function $g\from C\to \R$ and a point $x\in C$, and 
			returns a Busemann subgradient
			for $g$ at $x$:
			$[\xi,s] \leftarrow \texttt{Busemann}(g,x)$. If $[\xi,s]$ is a Busemann subgradient for $g$
			at $x$, we will use the notation $r_{x,\xi}$ to denote the ray 
			issuing from $x$ with direction $\xi$.
			\begin{customthm}{A}
				\label{assumptionA}~\\
				\vspace{-0.5cm}
			\begin{enumerate}[label=(\roman*)]
			\item $(X,d)$ is a Hadamard space with the geodesic extension property.
			\item $C\seq X$ is nonempty, closed, and geodesically convex.

			\item The function $f$ decomposes as $f = \sum_{i=1}^m f_i$, where each 
				$f_i, i = 1,\dots,m$ is Busemann subdifferentiable on an open set 
			containing $C$.
						\item The optimal set of $\inf_{x\in C} f(x)$ is nonempty, and denoted by $X^*$. The optimal value of the
				problem is denoted by $f_{\text{opt}}$.
			\item There is a constant $L \geq 0$ such that
				for all $i=1,\dots,m$, every Busemann subgradient $[\xi,s]$ for $f_i$ at every point in $C$
				has $s\leq L$.
\end{enumerate}
	\end{customthm}
		
We begin with a stochastic subgradient method.

	\begin{algorithm}[h!]
		\caption{Stochastic Busemann Subgradient Method}
			\begin{algorithmic}
				\Require $x^0 \in C, \st{t_k}_{k=0}^\infty \subset (0,+\infty), \st{i(k)}_{k=0}^\infty$ independent Uniform$\st{1,
				\dots,m}$ 
				\For{$k=0,1,2\dots$}
				\State $[\xi_{k},s_{k}] \leftarrow \texttt{Busemann}(f_{i(k)}, x^{k})$
				\State $x^{k+1} \leftarrow P_C\left(r_{x_{k},\xi_{k}}\left(s_{k}t_k\right)\right)$
				\EndFor 		
				\end{algorithmic}
\label{alg:stochsub}
		\end{algorithm}

		Understanding the complexity of Algorithm \ref{alg:stochsub} is our immediate goal,
		towards which the next lemma takes us most of the way.
				\begin{lemma}\textbf{(Projected Busemann subgradient inequality)}
			\label{lem:projsub}
			Suppose the function $f\from X\to \R$ is Busemann subdifferentiable on a nonempty, closed, and 
			geodesically convex set $C$. Let $x\in C$, $t > 0$ and choose a
			Busemann subgradient $[\xi,s]$ for $f$ at $x$.
			Define the new point
			\[x^+ = \begin{cases}
					P_C\left(r_{x,\xi}\left( s t\right)\right), & s > 0\\
					x, & s = 0,
		\end{cases}\]
		where $r_{x,\xi}$ is the ray issuing from $x$ with direction $\xi$.
	Then for any $y \in C$,
	\[d(x^{+},y)^2\leq d(x,y)^2 - 2t (f(x)- f(y)) +s^2t^2.\]
		\end{lemma}
		\begin{proof}
			For brevity we denote $r = r_{x,\xi}$. If $s = 0$ then the desired inequality reduces
			to $f(x)\leq f(y)$ for all $y\in C$, which holds because $x$ minimizes $f$ over $C$ 
			(cf.\ the discussion
				following Definition \ref{def:epihoro}). Thus we may assume
			$s > 0$.
			Since $X$ is Hadamard, $x\mapsto \frac{1}{2}d(x,y)^2$ is 1-strongly convex for any $y\in C$.
			It follows that for all
			$\delta \geq s t$ we have
			\begin{align*}
				d(x^{+},y)^2 &= d\left(P_C(r\hspace{-0.05cm}\left(st\right)),P_C(y)\right)^2\\
				&\leq d\left(r\hspace{-0.05cm}\left(st\right),y\right)^2\\
				&\leq \left(1-\frac{s t}{\delta}\right)d(x,y)^2 + \frac{s t}{\delta}
				d(r(\delta),y)^2 - \left(1 - \frac{s t}{\delta}\right)\frac{s t}{\delta}
				d(x,r(\delta))^2\\
				&= \left(1-\frac{s t}{\delta}\right)d(x,y)^2 + \frac{s t}{\delta}
				d(r(\delta),y)^2 + s^2t^2 - \delta s t\\
				&= \left(1-\frac{s t}{\delta}\right)d(x,y)^2 + \frac{s t}{\delta}(
			d(r(\delta),y)^2 - \delta^2) + 
			s^2t^2.
			\end{align*}
			The second line uses nonexpansivity of the projection $P_C$, see \cite[Theorem 2.1.12]{bacakbook}.
			Letting $\delta\to \infty$ and using once more \eqref{eqn:quadraticbuse} we deduce
						\[d(x^{+},y)^2 \leq d(x,y)^2 +  2 s t b_r(y) + s^2t^2.\]
			By (\ref{eqn:busesub}) we have
			$s b_r(y) 
			\leq -(f(x) - f(y))$.
			We conclude:
			\[d(x^{+},y)^2 \leq d(x,y)^2 +  2s t b_r(y) + s^2t^2 \leq
			d(x,y)^2 - 2 t(f(x)-f(y)) +s^2t^2.\]
		\end{proof}
		
The analysis of the stochastic subgradient method borrows from \cite[Lemma 3.6]{bacak-means} and 
			\cite[Theorem 8.35]{beckfirstorder}. Define the notation
			$f^k_{\text{best}} := \min_{i=1,\dots,k} f(x^i) - f_{\text{opt}}$.  For stochastic terminology, see \cite{resnick}.
			\begin{theorem}\textbf{(Complexity of stochastic Busemann subgradient method)}
			\label{thm:stochcomplexity}
			Suppose Assumption \ref{assumptionA} holds, and let $\st{t_k}_{k=0}^\infty$ be a sequence of positive 
			stepsizes and $\st{i(k)}_{k=0}^\infty$ a sequence of independent random variables distributed
			uniformly over $\st{1,\dots,m}$.
			Let $\st{x^k}_{k=0}^\infty$ be the sequence of iterates 
			generated by Algorithm \ref{alg:stochsub} with the above parameters.
			Suppose furthermore that the diameter of $C$ is bounded above by $D > 0$.
			If $t_k = \frac{D}{Lm\sqrt{k+1}}$ then for all $k\geq 2$,
			\[\mathbb{E}\left[\min_{i=1,\dots,k} f(x^i) - f_{\text{opt}}\right] \leq \frac{2(1+\log(3))mLD}{\sqrt{k+2}}.\]
			More generally, if $\sum_{k=0}^nt_k^2/\sum_{k=0}^nt_k \to 0$ as $n\to \infty$ then $\mathbb{E}[f^k_{\text{best}}] \to f_{\text{opt}}$ as $k\to \infty$ even if $C$ is unbounded.
	\end{theorem}
\begin{proof}
Denoting the nonnegative integers by $\N_0$, the underlying probability space is $\Omega = \st{1,\dots,m}^{\N_0}$ equipped with the product of the 
uniform $\st{1,\dots,m}$ probability measures.
At step $k \in \N_0$, choose a Busemann subgradient $[\xi_{k},s_{k}]$ for $f_{i(k)}$ at $x^k$. 
Then define $x^{k+1} := P_C(r_{x^k, \xi_{k}}(s_{k}t_k))$, as in Algorithm \ref{alg:stochsub}.
Consider a given (random) iterate $x^k$. 
Lemma \ref{lem:projsub} and Assumption \ref{assumptionA}(v) together imply
\[d(x^{k+1},x^*)^2 \leq d(x^k,x^*)^2 -2t_k(f_{i(k)}(x^k) - f_{i(k)}(x^*)) + L^2t_k^2.\]
Let $\fami_k = \sigma(x^0,\dots,x^k)$ be the $\sigma$-algebra generated by the variables $x^0,\dots,x^k$.
Taking the conditional expectation with respect to $\fami_k$ gives
\begin{align*}
	\E[d(x^{k+1},x^*)^2 \mid \fami_k] &\leq d(x^k,x^*)^2 -2t_k\E[f_{i(k)}(x^k)-f_{i(k)}(x^*) \mid \fami_k] + L^2t_k^2\\
&= d(x^k,x^*)^2 -\frac{2t_k}{m}\sum_{i=1}^m(f_{i}(x^k)-f_{i}(x^*)) + L^2t_k^2\\
&= d(x^k,x^*)^2 -\frac{2t_k}{m}(f(x^k)-f_{\text{opt}}) + L^2t_k^2.
\end{align*}
Now taking the expectation, we arrive at
\[\E[d(x^{k+1},x^*)^2] \leq \E[d(x^k,x^*)^2]-\frac{2t_k}{m} \E[f(x^k) - f_{\text{opt}}] + L^2t_k^2.\]
Summing from $k= j$ to $n\geq j$ and rearranging:
\[\frac{2}{m}\sum_{k=j}^n t_k \E[f(x^k)-f_{\text{opt}}] \leq \E[d(x^j,x^*)^2] + L^2\sum_{k=j}^nt_k^2.\]
By the elementary inequality $\E[\min_{k=j,\dots,n}f(x^k)] \leq \min_{k=j,\dots,n}\E[f(x^k)]$, we derive
\[\E\left[\min_{k=0,\dots,n} f(x^k) - f_{\text{opt}}\right] 
\leq \E\left[\min_{k=j,\dots,n}f(x^k) - f_{\text{opt}}\right] \leq \frac{m (\E[d(x^j,x^*)^2] + L^2\sum_{k=j}^nt_k^2)}{2\sum_{k=j}^n t_k}.\]
Finally, since the diameter of $C$ is at most $D$ we have $\E[d(x^j,x^*)^2] \leq D^2$. Choosing $t_k = \frac{D}{L\sqrt{k+1}}$, $j = \ceil{n/2}$,
and appealing to 
\cite[Lemma 8.27(b)]{beckfirstorder} gives 
\[\E\left[\min_{k=0,\dots,n} f(x^k) - f_{\text{opt}}\right]  \leq \frac{2(1+\log(3))mDL}{\sqrt{n+2}} ~ ~ ~ \text{ for all } n \geq 2.\]
The last statement of the theorem can be proven in exactly the same way
			 as \cite[Theorem 8.35(a)]{beckfirstorder}.
\end{proof}
A derandomized, incremental variant of Algorithm \ref{alg:stochsub} enjoys the same 
complexity bound. The procedure is detailed in Algorithm \ref{alg:incrsub}.
	\begin{algorithm}[h!]
			\caption{Incremental Busemann Subgradient Method}
			\begin{algorithmic}
				\Require $x^0 \in C, \st{t_k}_{k=0}^\infty \subset (0,+\infty)$
				\For{$k=0,1,2\dots$}
				\State $x^{k,0} \leftarrow x^k$
				\For{$i=0,1,\dots,m-1$}
				\State $[\xi_{k,i},s_{k,i}] \leftarrow \texttt{Busemann}(f_{i+1}, x^{k,i})$
				\State $x^{k,i+1} \leftarrow P_C\left(r_{x^{k,i},\xi_{k,i}}\left(s_{k,i}t_k\right)\right)$
				\EndFor
				\State $x^{k+1} \leftarrow x^{k,m}$
				\EndFor 		
				\end{algorithmic}
\label{alg:incrsub}
		\end{algorithm}
		The following lemma and its proof are straightforwardly adapted from \cite[Lemma 8.39]{beckfirstorder},
			which is itself an adaptation of the original 
		\cite[Lemma 2.1]{bertseknedic}.
		\begin{lemma}\textbf{(Incremental Busemann subgradient inequality)}
			\label{lem:incremental}
			Suppose Assumption \ref{assumptionA}  holds, and let $\st{x^k}_{k=0}^\infty$ be the sequence of iterates 
			generated by Algorithm \ref{alg:incrsub} with positive stepsizes $\st{t_k}_{k=0}^\infty$.
			Then for any $x^* \in X^*$ and $k\geq 0$,
			\[d(x^{k+1},x^*)^2 \leq d(x^k,x^*)^2 - 2t_k(f(x^k) - f_{\text{opt}}) + t_k^2m^2L^2.\]
		\end{lemma}
		\begin{proof}
			Fix $i\in \st{0,1,\dots,m-1}$. Lemma \ref{lem:projsub} proves
			\[d(x^{k,i+1},x^*)^2 \leq
			d(x^{k,i},x^*)^2 - 2t_k(f_{i+1}(x^{k,i})-f_{i+1}(x^*)) + s_{k,i}^2t_k^2.\]
			Summing the inequality over $i=0,1,\dots,m-1$ and using the identities $x^{k,0} = x^k, x^{k,m} = x^{k+1}$
			we deduce:
			\begin{align}
				\label{eqn:firstestims}
				d(x^{k+1},x^*)^2 &\leq d(x^{k},x^*)^2 - 2t_k\sum_{i=0}^{m-1}(f_{i+1}(x^{k,i})-f_{i+1}(x^*)) + t_k^2
				\sum_{i=0}^{m-1}s_{k,i}^2 \nonumber \\
			&\leq d(x^{k},x^*)^2 - 2t_k\sum_{i=0}^{m-1}(f_{i+1}(x^{k,i})-f_{i+1}(x^*)) + t_k^2mL^2 \nonumber \\
			&=\! d(x^{k},x^*)^2 - 2t_k\left(f(x^k) - f_{\text{opt}} + \sum_{i=0}^{m-1}(f_{i+1}(x^{k,i})-f_{i+1}(x^k))\right) + t_k^2mL^2 \nonumber \\
			&\leq d(x^{k},x^*)^2 - 2t_k(f(x^k) - f_{\text{opt}}) + 2t_kL\sum_{i=0}^{m-1}d(x^{k,i},x^k) + t_k^2mL^2.
		\end{align}
		The second inequality uses Assumption \ref{assumptionA}(v), and the last
		inequality makes use of $L$-Lipschitz continuity via Proposition \ref{prop:liphoro}.	
		We aim to control the size of $d(x^{k,i},x^k), i=0,\dots,m-1$ in the last line above, so we start by estimating
		\begin{equation}
			\label{eqn:trianglelem}
			d(x^{k,1},x^k) = d(P_C(r_{\xi_{k,0}}(s_{k,0}t_k)),P_C(x^k)) \leq d(r_{\xi_{k,0}}(s_{k,0}t_k),x^k) = s_{k,0}t_k \leq  L t_k.
		\end{equation}
		This required nonexpansivity of the projection $P_C$ as well as Assumption \ref{assumptionA}(v).
		Moving on to $x^{k,2}$, we argue similarly
		\[d(x^{k,2},x^k) \leq d(r_{\xi_{k,1}}(s_{k,1}t_k),x^k) \leq 
		d(r_{\xi_{k,1}}(s_{k,1}t_k),x^{k,1}) + d(x^{k,1},x^k) \leq 2Lt_k.\]
		Here we used nonexpansivity of $P_C$, the triangle inequality, Assumption \ref{assumptionA}(v),
		and \eqref{eqn:trianglelem}. Iterating these estimates for $i=2,\dots, m-1$ we conclude 
			\[d(x^{k,i},x^k)\leq iLt_k ~ ~ ~ i = 0,1,\dots, m-1.\]
			Combined with \eqref{eqn:firstestims} we find:
			\begin{align*}
		d(x^{k+1},x^*)^2&\leq d(x^{k},x^*)^2 - 2t_k(f(x^k) - f_{\text{opt}}) + 2t_kL\sum_{i=0}^{m-1}d(x^{k,i},x^k) + t_k^2mL^2\\
			&\leq d(x^{k},x^*)^2 - 2t_k(f(x^k) - f_{\text{opt}}) + 2t_k^2L^2\sum_{i=0}^{m-1}i + t_k^2mL^2\\
			&=  d(x^{k},x^*)^2 - 2t_k(f(x^k) - f_{\text{opt}}) + t_k^2L^2m^2.
			\end{align*}
		\end{proof}
		
		With Lemma \ref{lem:incremental} in hand, the proof below
		is a rewrite of \cite[Theorem 8.40]{beckfirstorder}
		with the metric $d$ in place of the Euclidean distance.

		\begin{theorem}\textbf{(Complexity of incremental Busemann subgradient method)}
			\label{thm:complexity}
			Suppose Assumption \ref{assumptionA} holds, and let $\st{x^k}_{k=0}^\infty$ be the sequence of iterates 
			generated by Algorithm \ref{alg:incrsub} with positive stepsizes $\st{t_k}_{k=0}^\infty$.
			Suppose furthermore that the diameter of $C$ is bounded above by $D > 0$.
			If $t_k = \frac{D}{Lm\sqrt{k+1}}$ then for all $k\geq 2$,
			\[f^k_{\text{best}}  \leq \frac{2(1+\log(3))mLD}{\sqrt{k+2}}.\]
			More generally, if $\sum_{k=0}^nt_k^2/\sum_{k=0}^nt_k \to 0$ as $n\to \infty$ then $f^k_{\text{best}} \to f_{\text{opt}}$ as $k\to \infty$ even if $C$ is unbounded.
		\end{theorem}
		\begin{proof}
			By Lemma \ref{lem:incremental}, for any $n\geq 0$
			\[d(x^{n+1},x^*)^2 \leq d(x^n,x^*)^2 - 2t_n(f(x^n) - f_{\text{opt}}) + t_n^2L^2m^2.\]
			Summing over $n= \ceil{k/2},\dots,k$ we find
			\[d(x^{k+1},x^*)^2 \leq d(x^{\ceil{k/2}},x^*)^2 - 2\sum_{n=\ceil{k/2}}^kt_n(f(x^n)-f_{\text{opt}}) 
			+L^2m^2\sum_{n=\ceil{k/2}}^kt_n^2.\]
			Rearranging gives
			\[2\sum_{n=\ceil{k/2}}^kt_n(f(x^n)-f_{\text{opt}}) \leq d(x^{\ceil{k/2}},x^*)^2 +L^2m^2\sum_{n=\ceil{k/2}}^kt_n^2.\]
			We readily estimate
			\[\min_{i=1,\dots,k}f(x^i) - f_{\text{opt}} \leq \frac{d(x^{\ceil{k/2}},x^*)^2 +L^2m^2\sum_{n=\ceil{k/2}}^kt_n^2}{2\sum_{n=\ceil{k/2}}^k t_n} \leq \frac{D^2 +L^2m^2\sum_{n=\ceil{k/2}}^kt_n^2}{2\sum_{n=\ceil{k/2}}^k t_n}.
			\]
			Plugging in 
			$t_n = D/(Lm\sqrt{n+1})$ we arrive at
			\[\min_{i=1,\dots,k}f(x^i) - f_{\text{opt}} \leq 
				\frac{mLD}{2} \frac{\left(1+\sum_{n=\ceil{k/2}}^k\frac{1}{n+1}\right)}{\sum_{n=\ceil{k/2}}^k \frac{1}{\sqrt{n+1}}}.
			\]
			Applying \cite[Lemma 8.27(b)]{beckfirstorder} gives the bound \[\frac{\left(1+\sum_{n=\ceil{k/2}}^k\frac{1}{n+1}\right)}{\sum_{n=\ceil{k/2}}^k \frac{1}{\sqrt{n+1}}} \leq \frac{4(1+\log(3))}{\sqrt{k+2}}\] and
			the final estimate follows. The last statement of the theorem can be proven in exactly the same way
			 as \cite[Theorem 8.40(a)]{beckfirstorder}.
		\end{proof}
	\end{section}
	\begin{section}{Subgradient methods for the median problem}
		On any Hadamard space $(X,d)$ with the geodesic extension property, we specialize to the \textit{median problem}
		\begin{equation}
			\label{eqn:median}
			\min\st{f(x) := \sum_{i=1}^mw_id(x,a_i) \mid x\in X},
		\end{equation}
		where $\aaa = \st{a_1,\dots,a_m} \seq X$ are given points and $w\in \R^m_+$ is a vector of nonnegative weights summing to one.
		It is well-known that problem (\ref{eqn:median}) admits at least one minimizer.
		The functions $f_i := w_i d(x,a_i), i = 1,\dots,m$ share a common Lipschitz constant of $w^* := \max_i w_i >0$,
		so Assumption \ref{assumptionA}(v) is satisfied thanks to Proposition 
		\ref{prop:liphoro}(ii).
		Each $f_i$ is Busemann subdifferentiable by Example \ref{ex:distp}, 
		with a Busemann subgradient $[r_{x,a_i}(\infty),w_i]$ at $x \neq a_i$, where $r_{x,a_i}$ denotes any ray originating at $x$ and passing through $a_i$,  and $[0]$ at $x=a_i$.
		To attain the stronger theoretical complexity guarantee in Theorem \ref{thm:complexity} we require a bound on the 
		diameter of the feasible region. The structure of the problem implies that the minimizers cannot be too far from
		points in $\aaa$. Let $a^*$ be the point in $\aaa$ corresponding to the coefficient
		$w^*$. Then for any minimizer $x^* \in X^*$ and any $x^0\in X$, we have
		\[w^*d(x^*,a^*) \leq \sum_{i=1}^mw_id(x^*,a_i) \leq \sum_{i=1}^mw_id(x^0,a_i) = f(x^0).\]
		It follows that $X^* \seq B_{f(x^0)/w^*}(a^*)$. This ball has diameter $D = 2f(x^0)/w^*$, and 
		projecting onto a ball is straightforward. 
		With Assumption \ref{assumptionA} verified,
		we can use our work above and the stepsize from Theorem \ref{thm:complexity} to specialize Algorithms 
		\ref{alg:stochsub} and \ref{alg:incrsub}
		to the median problem.
\begin{algorithm}
	\caption{Stochastic-Subgradient Median Algorithm}
			\begin{algorithmic}
				\Require $x^0 \in X$, $\st{i(k)}_{k=0}^\infty$ independent Uniform $\st{1,\dots,m}$
				\For{$k=0,1,2\dots$}
				\State $x^{k+1} \leftarrow 
				P_{B_{\frac{f\left(x^0\right)}{w^*}}(a^*)}\left(r_{x^{k},a_{i(k)}}\left(\frac{2w_{i(k)}f\left(x^0\right)}{w^*m\sqrt{k+1}}\right)\right)$
			\EndFor 		
				\end{algorithmic}
\label{alg:stochmedalg}
		\end{algorithm}

\begin{algorithm}
				\caption{Incremental Median Algorithm}
			\begin{algorithmic}
				\Require $x^0 \in X$
				\For{$k=0,1,2\dots$}
				\State $x^{k,0} \leftarrow x^k$
				\For{$i=0,1,\dots,m-1$}
				\State $x^{k,i+1} \leftarrow 
				P_{B_{\frac{f\left(x^0\right)}{w^*}}(a^*)}\left(r_{x^{k,i},a_{i+1}}\left(\frac{2w_{i+1}f\left(x^0\right)}{w^*m\sqrt{k+1}}\right)\right)$
				\EndFor
				\State $x^{k+1} \leftarrow x^{k,m}$
			\EndFor 		
				\end{algorithmic}
\label{alg:medalg}
		\end{algorithm}

		\begin{corollary}\textbf{(Median complexity)}
			\label{cor:medcomplex}
			Algorithm \ref{alg:stochmedalg} satisfies the following complexity bound for all $k\geq 2$:
			\[\mathbb{E}[f^k_{\text{best}}] \leq \frac{4(1+\log(3))mf\left(x^0\right) }{\sqrt{k+2}}.\]
			Algorithm \ref{alg:medalg} satisfies the same bound without the expectation.
		\end{corollary}
		\begin{proof}
			Set $D =  2f(x^0)/w^*, L = w^*$ in Theorems \ref{thm:stochcomplexity} and \ref{thm:complexity}.
	\end{proof}
	
	An analogous approach applies to the problem of computing $p$-means, and more generally to objectives of the form $f(x) = \sum_{i=1}^m \phi_i(d(x,a_i))$, where each function $\phi_i \from \R \to \R$ is nondecreasing and convex, covering the additional case of Huber estimators \cite{Schotz2025}.  The requisite tools are Proposition \ref{prop:chain} (the chain rule) and Proposition \ref{distanceprop} (subgradients of distance functions).  The corresponding versions of Algorithms \ref{alg:stochsub} and \ref{alg:incrsub} can be written explicitly in a fashion analogous to Algorithms \ref{alg:stochmedalg} and \ref{alg:medalg}.  

\subsection*{Comparison with cyclic proximal point method}
We next compare our subgradient-based methods with the proximal point methods described in \cite{bacak-means}.  One distinction that we note at the outset is the constraint set $C$ in the subgradient-based methods, allowing us to model hard constraints by incorporating projections.  The analysis presented in \cite{bacak-means} does not cover this ingredient.

Focusing specifically on the median problem, one proximal point method from \cite{bacak-means} is Algorithm \ref{alg:proxalg}.

\begin{algorithm}
	\caption{Cyclic Proximal Median Algorithm \cite{bacak-means}}
\begin{algorithmic}
	\Require $x^0 \in X, \st{t_k}_{k=0}^\infty \subset (0,+\infty)$
\For{$k=0,1,2\dots$}
\State $x^{k,0} \leftarrow x^k$
\For{$i=0,1,\dots,m-1$}
\State $x^{k,i+1} \leftarrow 
r_{x^{k,i},a_{i+1}}\left(\min\st{d(x^{k,i},a_{i+1}),w_{i+1}t_k}\right)$
\EndFor
\State $x^{k+1} \leftarrow x^{k,m}$
\EndFor 		
\end{algorithmic}
\label{alg:proxalg}
\end{algorithm}

\noindent
When $X$ is proper, \cite{bacak-means} shows that for any sequence of positive stepsizes satisfying 
		\begin{equation}
			\label{eqn:rm}
			\sum_{k=0}^\infty t_k = \infty, ~ ~ ~\sum_{k=0}^\infty t_k^2 < \infty
		\end{equation}
(a choice motivated by stochastic variants),
the iterates $\st{x^k}_{k=0}^\infty$ generated by Algorithm \ref{alg:proxalg} converge to a median of $\aaa$. 
Although \cite{bacak-means} omits a formal complexity statement, the proof implies an $O(\eps^{-2})$ complexity
bound for finding $\eps$-minimizers (in terms of function values).

		Algorithm \ref{alg:proxalg} resembles a version of Algorithm \ref{alg:incrsub} with $C=X$.
The main difference is that the former requires only geodesic segments instead of rays, and indeed,
Algorithm \ref{alg:proxalg} does not rely on the geodesic extension property. 
		If the iterates remain bounded away from the set $\aaa$, the algorithms eventually resemble each other, because the stepsizes decay to zero, ensuring that each update $x^{k,i+1}$ lies in the geodesic segment $[x^{k,i},a_i]$, as we observe empirically.  The following example shows that this behavior can fail if the set $\aaa$ contains its median. 
		
		\begin{example}
			\label{ex:raysneeded}
			Consider the case $C=X =\R$ and $\aaa = \st{0}$. Algorithm \ref{alg:incrsub} amounts to the classical subgradient method applied to the function $|\cdot|$.
			Taking $x^0 = 1$ and $t_k = \frac{1}{k+2} + \frac{1}{k+1}$, Algorithm \ref{alg:incrsub} generates the sequence $x^k = \frac{(-1)^k}{k+1}$, overshooting the solution $0$ at each iteration, and hence requiring a ray oracle. Algorithm \ref{alg:proxalg}, on the other hand, generates a sequence in the interval $[0,1]$ monotonically decreasing to $0$, using the geodesic oracle.
		\end{example}
		
\noindent
Notwithstanding such examples, practical comparison (Example \ref{ex:ex1}) confirms our expectation that, for typical median problems the subgradient and proximal point methods (in this case Algorithms \ref{alg:medalg} and \ref{alg:proxalg}) are comparable.

\subsection*{Subgradient versus proximal methods:  a general comparison}

While Busemann-subgradient-based and proximal algorithms for computing $p$-means are comparable,
we point out several distinct advantages of the Busemann subgradient philosophy on more general problems.

First, each step in a subgradient method has a simple concrete interpretation in the underlying Hadamard space:  the step involves traversing a ray at a certain speed.  For example, given a smooth function on Euclidean space, that speed is just the slope of the function.  By contrast, proximal steps, being defined via a subproblem (\ref{eqn:proximal-update}), are more opaque.

Secondly, while we can compute proximal updates for a few simple instances like distance functions $d(a,\cdot)$, such examples are rare.  By contrast, subgradients are often available.  For example, the proximal update for a function of the form 
\[
x \mapsto \max_{a \in \aaa} d(a,x),
\]
for a nonempty finite set $\aaa \subset X$, seems expensive to approximate, whereas subgradients are easy to calculate via Remark \ref{rem:max}.  Analogously, any function expressed explicitly as a Busemann envelope has readily available subgradients.  Such functions might be useful as regularizers:  one well-known class are the sparsity-inducing polyhedral regularizers common in Euclidean optimization \cite{Vaiter2013}.

Lastly, as we have already noted, Busemann subgradient algorithms handle a nontrivial feasible region $C$ routinely, assuming that the projection $P_C$ is accessible.  Solving an optimization problem of the form $\inf_C \sum_{i=1}^m f_i$ using the proximal methods of \cite{bacak-means}, on the other hand, forces a choice of approaches, all with disadvantages.  One might model the hard constraint $x \in C$ using the indicator function $\delta_C$ as another summand, but the developments of \cite{bacak-means,OhtaPalfia2015} ostensibly concern only Lipschitz components.  An exact penalty function approach \cite[Proposition 6.3]{Clarke1998} lets us replace $\delta_C$ by a multiple of the distance function $\mbox{dist}_C$, which is Lipschitz.  However, even if all the components $f_i$ are $L$-Lipschitz, the required multiple must exceed $mL$, worsening the complexity estimates in comparison with the subgradient approach.  Finally, we might replace proximal steps for each $f_i$ with those for $f_i + \delta_C$, but unlike the unconstrained case, these steps may be hard to compute even when each $f_i$ is a distance function.

In summary, proximal algorithms are often hard to interpret and implement.  Stochastic and incremental subgradient methods provide a compelling alternative.
	\end{section}

	\begin{section}{Computational experiments}
		\label{sec:comp}
		We consider the BHV tree space $\mathcal{T}_n$ of binary trees on $n$ labelled leaves \cite{billeratree}. There are $(2n-3)!!$ such binary trees, each with $n-2$ internal edges.
		The space $\mathcal{T}_n$ models all such binary trees by ascribing an $(n-2)$-dimensional orthant $[0,\infty)^{n-2}$
		to each tree so that a point in each orthant describes a particular binary tree topology with a prescribed choice
		of nonnegative internal edge lengths. Tree space has proved an interesting model for
		comparison and averaging of phylogenetic trees. The space $\mathcal{T}_n$ was shown in 
		\cite{billeratree} to be CAT(0), is complete, and has the geodesic extension property, by \cite[Proposition II.5.10]{bridson} and \cite[p.\ 743]{billeratree}.

As a simple example, we focus on the space $\mathcal{T}_4$, where the binary trees have 4 labelled leaves.
Extending the computational examples below to $\mathcal{T}_n$ is straightforward: the only input to the code we use is a list of trees in \textit{Newick notation} \cite{archie1986newick}.
		$\mathcal{T}_4$ is a union of 15 two-dimensional quadrants.
		Our experiment uses Algorithms \ref{alg:stochmedalg} and \ref{alg:medalg} to estimate the median of a finite set of trees in $\mathcal{T}_n$, using the existing software package \texttt{SturmMean} \cite{sturmmean}, a polynomial-time algorithm \cite{treegeo} to compute geodesics in tree space. 
		Algorithms \ref{alg:stochmedalg} and \ref{alg:medalg} actually rely on rays which could extend beyond the span of a geodesic between two trees.  Such rays may not be unique,
and \texttt{SturmMean} does not compute them, but our examples involve stepsizes small enough to ensure that no extension is needed.  Unless otherwise stated, the initial point $x^0$ is the origin of tree space.
		
		\begin{example}
			\label{ex:ex1}
			Our first example comes from the documentation of \texttt{SturmMean} \cite{sturmmean}.
			This example is convenient because the trees embed simply in $\R^2$ in such a way that we can calculate
		the true median exactly, allowing us to demonstrate convergence of the best found function value
		to the optimal value. 
		Figure \ref{fig:trees} shows the three trees we consider, with equal weights
		$w_i = 1/3$ assigned to each tree.
		\begin{figure}
			
		\begin{center}

\tikzset{every picture/.style={line width=0.75pt}} 

\begin{tikzpicture}[x=0.75pt,y=0.75pt,yscale=-1,xscale=1,scale=0.9]

\draw  [color={rgb, 255:red, 0; green, 0; blue, 0 }  ,draw opacity=0 ][fill={rgb, 255:red, 0; green, 0; blue, 0 }  ,fill opacity=0.37 ] (160,60) -- (310,60) -- (310,210) -- (160,210) -- cycle ;
\draw    (310,210) -- (309.01,371) ;
\draw [shift={(309,373)}, rotate = 270.35] [color={rgb, 255:red, 0; green, 0; blue, 0 }  ][line width=0.75]    (10.93,-3.29) .. controls (6.95,-1.4) and (3.31,-0.3) .. (0,0) .. controls (3.31,0.3) and (6.95,1.4) .. (10.93,3.29)   ;
\draw    (310,210) -- (496,210.99) ;
\draw [shift={(498,211)}, rotate = 180.3] [color={rgb, 255:red, 0; green, 0; blue, 0 }  ][line width=0.75]    (10.93,-3.29) .. controls (6.95,-1.4) and (3.31,-0.3) .. (0,0) .. controls (3.31,0.3) and (6.95,1.4) .. (10.93,3.29)   ;
\draw    (310,210) -- (130,209.8) ;
\draw [shift={(128,209.8)}, rotate = 0.06] [color={rgb, 255:red, 0; green, 0; blue, 0 }  ][line width=0.75]    (10.93,-3.29) .. controls (6.95,-1.4) and (3.31,-0.3) .. (0,0) .. controls (3.31,0.3) and (6.95,1.4) .. (10.93,3.29)   ;
\draw    (310,210) -- (310,42) ;
\draw [shift={(310,40)}, rotate = 90] [color={rgb, 255:red, 0; green, 0; blue, 0 }  ][line width=0.75]    (10.93,-3.29) .. controls (6.95,-1.4) and (3.31,-0.3) .. (0,0) .. controls (3.31,0.3) and (6.95,1.4) .. (10.93,3.29)   ;
\draw  [fill={rgb, 255:red, 0; green, 0; blue, 0 }  ,fill opacity=1 ] (190.05,227.47) .. controls (190.08,226.07) and (191.24,224.96) .. (192.64,224.99) .. controls (194.03,225.02) and (195.14,226.18) .. (195.11,227.58) .. controls (195.08,228.97) and (193.93,230.08) .. (192.53,230.05) .. controls (191.13,230.02) and (190.02,228.87) .. (190.05,227.47) -- cycle ;
\draw  [fill={rgb, 255:red, 0; green, 0; blue, 0 }  ,fill opacity=1 ] (349.2,132.4) .. controls (348.76,133.73) and (347.33,134.45) .. (346.01,134.01) .. controls (344.68,133.57) and (343.96,132.14) .. (344.4,130.81) .. controls (344.84,129.48) and (346.27,128.76) .. (347.6,129.2) .. controls (348.92,129.64) and (349.64,131.08) .. (349.2,132.4) -- cycle ;
\draw  [fill={rgb, 255:red, 0; green, 0; blue, 0 }  ,fill opacity=1 ] (390.8,267.6) .. controls (390.36,268.92) and (388.92,269.64) .. (387.6,269.2) .. controls (386.27,268.76) and (385.55,267.33) .. (385.99,266.01) .. controls (386.43,264.68) and (387.86,263.96) .. (389.19,264.4) .. controls (390.52,264.84) and (391.24,266.27) .. (390.8,267.6) -- cycle ;
\draw    (446.8,48.8) -- (446.8,84.8) ;
\draw    (446.8,84.8) -- (424.8,108.8) ;
\draw    (446.8,84.8) -- (497.8,132.8) ;
\draw    (424.8,108.8) -- (409.8,137.8) ;
\draw    (424.8,108.8) -- (437.8,137.8) ;
\draw    (497.8,132.8) -- (483.8,159.8) ;
\draw    (497.8,132.8) -- (511.3,159.8) ;
\draw    (152.8,239.8) -- (152.8,275.8) ;
\draw    (152.8,275.8) -- (139.8,286.8) ;
\draw    (212.8,368.6) -- (236,397) ;
\draw    (139.8,286.8) -- (116.8,307.6) ;
\draw    (139.8,286.8) -- (212.8,368.6) ;
\draw    (212.8,368.6) -- (186,397) ;
\draw    (152.8,275.8) -- (172.8,299.8) ;
\draw    (479.8,253.6) -- (479.8,289.6) ;
\draw    (479.8,289.6) -- (432.8,317.6) ;
\draw    (367.8,359.2) -- (432.8,317.6) ;
\draw    (367.8,359.2) -- (347.8,381.6) ;
\draw    (367.8,359.2) -- (384.8,382.6) ;
\draw    (479.8,289.6) -- (498.8,315) ;
\draw    (432.8,317.6) -- (448.8,342) ;

\draw (194.58,230.92) node [anchor=north west][inner sep=0.75pt]    {$( -3,-1/2)$};
\draw (351.2,135.8) node [anchor=north west][inner sep=0.75pt]    {$( 1,2)$};
\draw (390.39,270.2) node [anchor=north west][inner sep=0.75pt]    {$( 2,-3/2)$};

\draw (405,139.4) node [anchor=north west][inner sep=0.75pt]    {$a$};
\draw (433,139.4) node [anchor=north west][inner sep=0.75pt]    {$b$};
\draw (478,160.4) node [anchor=north west][inner sep=0.75pt]    {$c$};
\draw (507,160.4) node [anchor=north west][inner sep=0.75pt]    {$d$};

\draw (112,309.4) node [anchor=north west][inner sep=0.75pt]    {$a$};
\draw (182,397.4) node [anchor=north west][inner sep=0.75pt]    {$b$};
\draw (232,397.4) node [anchor=north west][inner sep=0.75pt]    {$c$};
\draw (170,299.4) node [anchor=north west][inner sep=0.75pt]    {$d$};

\draw (344,382.2) node [anchor=north west][inner sep=0.75pt]    {$a$};
\draw (381,383.2) node [anchor=north west][inner sep=0.75pt]    {$b$};
\draw (444.8,342) node [anchor=north west][inner sep=0.75pt]    {$c$};
\draw (493.8,315) node [anchor=north west][inner sep=0.75pt]    {$d$};

\end{tikzpicture}

		\end{center}
		\caption{Three trees in $\mathcal{T}_4$ with neighboring respective orthants, embedded isometrically in $\R^2$
		(recreated from \cite{sturmmean}).}
\label{fig:trees}
\end{figure}
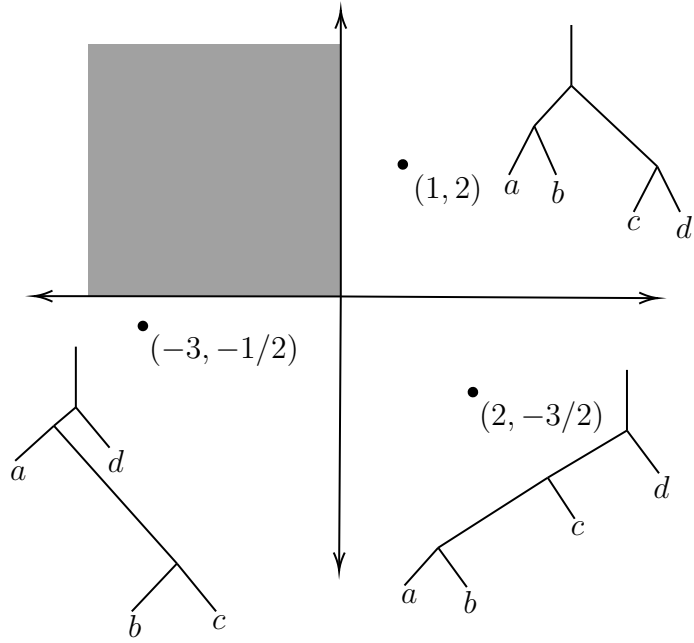
		For $x$ in the top right quadrant we have an explicit representation for $f$:
		\[f(x) = \frac{1}{3}\left(\|x - (1,2)\| + \|x - (2,-3/2)\| + \|x\| + \frac{\sqrt{37}}{2}\right).\]
		From here it is easy to check that 
		\[x^* = \left(\frac{2657-1038\sqrt{3}}{1898},\frac{3006 - 1369\sqrt{3}}{5694}\right)\]
		satisfies the optimality condition $\grad f(x^*) = 0$. Hence $x^*$ is a local minimizer of $f$, 
		and the convexity of $f$ on $\mathcal{T}_4$ implies $x^*$ is a global minimizer.
		Thus we set 
		\[f_{\text{opt}} = f(x^*) = \frac{1}{3}\sqrt{\frac{1}{2}\left(43+11\sqrt{3}+\sqrt{37(49+22\sqrt{3})}\right)}.\]
		The result of $10^5$ iterations of
		Algorithms \ref{alg:stochmedalg} and \ref{alg:medalg} are shown in Figure \ref{fig:medianplot}. We also tried initial points far from the median, having tree topologies not corresponding to one of the quadrants containing the given trees and with larger branch lengths. 
		\begin{figure}
			\begin{center}
				\includegraphics[scale=0.75]{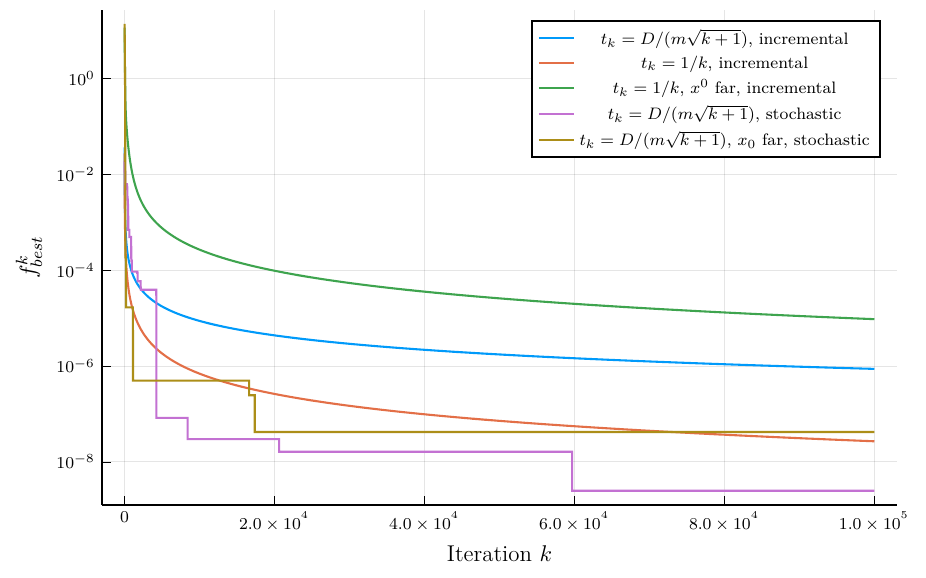}
			\caption{Convergence of $f(x^k)$ to $f_{\text{opt}}$ in Example \ref{ex:ex1} using different stepsizes
			and choices of initial tree.}
\label{fig:medianplot}
		\end{center}
		\end{figure}
		
				We illustrate Algorithm \ref{alg:medalg} in Figure \ref{fig:medianplot}. 
		Under the regime $C = X$ and $t_k = 1/k$, our experiments indicated that
		this algorithm coincides with Algorithm \ref{alg:proxalg} (in line with our earlier discussion), so to avoid confusion, we only show results for the former.  In spite of the stronger theoretical complexity guaranteed
		by the choice of stepsize in Theorem \ref{thm:complexity}, the stepsize $t_k = 1/k$ seems to achieve better
		performance based on the plots in Figures \ref{fig:medianplot} and \ref{fig:medianplot2}. 
		This behavior specifically relates
		to the incremental subgradient algorithm; our experiments showed that the stochastic algorithm performed
		poorly with stepsizes $t_k=1/k$ (such results were omitted from the plots), whereas
		its performance was very competitive using $t_k \sim 1/\sqrt{k}$.  

	\end{example}

	\begin{example}
		\label{ex:ex2}
		Now we consider three trees in $\mathcal{T}_4$ whose median lies on the common boundary ray 
		of three quadrants, with each quadrant containing one of the trees. We refer to this common boundary ray as the spine.
		A figure illustrating the local geometry of this setup
		can be found in \cite[Figure 10]{billeratree}. The three neighboring quadrants arise by permuting the labels on the
		leaves of the left subtree of the leftmost tree in Figure \ref{fig:treetop}. The spine corresponds to the second tree in this 
		same figure, obtained by contracting the lower internal edge so that the left subtree leaves become siblings. 
		The label $L$ on the separated leaf remains fixed throughout.

		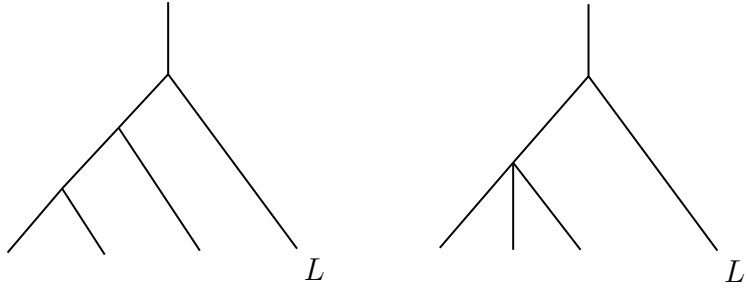
\begin{figure}
			\begin{center}

\tikzset{every picture/.style={line width=0.75pt}} 

\begin{tikzpicture}[x=0.75pt,y=0.75pt,yscale=-1,xscale=1,scale=0.8]

\draw    (229.8,95.2) -- (229.8,131.6) ;
\draw    (229.8,131.6) -- (204.8,158.6) ;
\draw    (176.32,189.28) -- (148.8,221.6) ;
\draw    (176.32,189.28) -- (197.8,222.6) ;
\draw    (229.8,131.6) -- (294.8,219.6) ;
\draw    (204.8,158.6) -- (245.8,220.6) ;
\draw    (176.32,189.28) -- (204.8,158.6) ;
\draw    (441.8,96.2) -- (441.8,132.6) ;
\draw    (441.8,132.6) -- (403.8,176.2) ;
\draw    (403.8,176.2) -- (366.8,219.2) ;
\draw    (403.8,176.2) -- (403.8,220.2) ;
\draw    (441.8,132.6) -- (506.8,220.6) ;
\draw    (403.8,176.2) -- (437.8,220.2) ;

\draw (296.8,223) node [anchor=north west][inner sep=0.75pt]    {$L$};
\draw (508.8,224) node [anchor=north west][inner sep=0.75pt]    {$L$};

\end{tikzpicture}

\caption{The tree topology defining three neighboring quadrants and their common spine.}
			\label{fig:treetop}
		\end{center}
		\end{figure}

		The median landing on a negligible subset illustrates the well-known phenomenon of \textit{stickiness\/}: in certain cubical complexes,  the mean of a randomly generated set of points will lie on a lower-dimensional face with positive probability
		\cite{stickylim}. We observe similarly with the median: letting $p_i, i= 1,2,3$ denote the point $(1,1)$ in 
		some ordering of the three neighboring quadrants, the median clearly lies on the spine, and
		one can show
		that for $\eps = 1/(1+2\sqrt{3})$ the median remains stuck on the spine even after perturbing 
		each of the $p_i$ within the box $p_i + [-\eps,\eps]^2$ in their respective quadrants. In this way we obtain three points 
		whose median lies on the spine:
		\[a = \left(\frac{2\sqrt{3}}{1+2\sqrt{3}}, \frac{2\sqrt{3}}{1+2\sqrt{3}}\right), ~ ~ ~ b = (1,1),
		~ ~ ~ c = \left(\frac{2+2\sqrt{3}}{1+2\sqrt{3}},\frac{2+2\sqrt{3}}{1+2\sqrt{3}}\right).\]
		Note that $a,b,c$ correspond directly to trees with the given internal branch lengths and the three tree topologies
		implicit in Figure \ref{fig:treetop}.
		The height of the median on the spine can be calculated as the optimal solution to the following problem:
		\[\min_y \sqrt{(y-a_2)^2 +a_1^2} + \sqrt{(y-b_2)^2+b_1^2} + \sqrt{(y-c_2)^2 + c_1^2}.\]
		Numerically, we find the optimal value to be $y^* \approx 0.9668\ldots$. 
		This gives an estimate for the optimal value:  $f_{\text{opt}} \approx 1.0168$.
		Initialized at the origin, the results of running Algorithms \ref{alg:stochmedalg} and \ref{alg:medalg}
		on this problem for $10^5$ 
		iterations are shown in Figure \ref{fig:medianplot2}. We also tested the algorithms initialized at
		a point $x^0$ far from the median, in the same sense as Example \ref{ex:ex1}. 
		Reflecting on both the experimental results of both examples, we can appreciate the efficiency of the
		stochastic algorithm;  it 
		keeps pace with the incremental algorithm using only $1/3$ of the Busemann subgradient oracle calls.
	 
		\begin{figure}
			\begin{center}
			\includegraphics[scale=0.75]{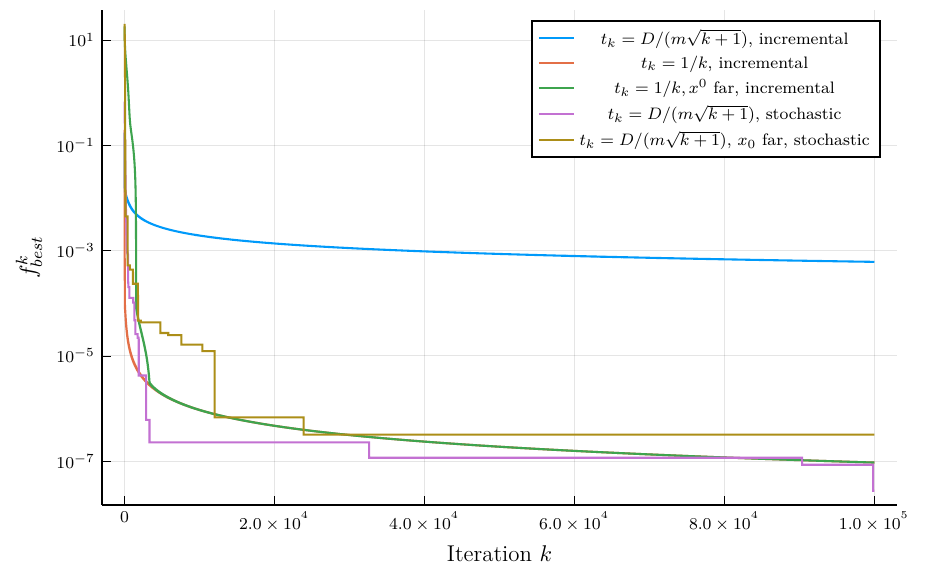}
			\caption{Convergence of $f(x^k)$ to $f_{\text{opt}}$ in Example \ref{ex:ex2} using different
			stepsizes and choices of initial tree.}
		\label{fig:medianplot2}
		\end{center}
		\end{figure}

	\end{example}
	
\bigskip

\noindent
{\bf\large Acknowledgement:}  The authors thank two anonymous reviewers for many helpful and substantial suggestions that greatly improved this work.

	\end{section}
		
	\bibliographystyle{plain}
	\small
\parsep 0pt

\begin{thebibliography}{10}

\bibitem{Afsari2012}
	B.~Afsari, R.~Tron and R.~Vidal.
\newblock On the convergence of gradient descent for finding the {R}iemannian center of mass.
\newblock {\em SIAM Journal on Control and Optimization}, 51(3):2230-2260, 2013.
	

\bibitem{manifoldquasi}
Q.~H. Ansari and M.~Uddin.
\newblock Convergence analysis of incremental quasi-subgradient method on
  {R}iemannian manifolds with lower bounded curvature.
\newblock {\em Optimization}, 2024.

\bibitem{archie1986newick}
J.~Archie, W.H.E. Day, J.~Felsenstein, W.~Maddison, C.~Meacham, F.J. Rohlf, and
  D.~Swofford.
\newblock The {N}ewick tree format.
\newblock {\em More information:
  http://evolution.genetics.washington.edu/phylip/newicktree.html}, 1986.

\bibitem{arnaudon2013}
M.~Arnaudon and F.~Nielsen.
\newblock On approximating the {R}iemannian 1-center.
\newblock {\em Computational Geometry}, 46(1):93--104, 2013.

\bibitem{Ballmann1995Lectures}
W.~Ballmann.
\newblock {\em Lectures on Spaces of Nonpositive Curvature}, volume~25 of {\em
  Oberwolfach Seminars}.
\newblock Birkh{\"a}user Basel, 1995.

\bibitem{bacak-means}
M.~Ba\v{c}\'{a}k.
\newblock Computing medians and means in {H}adamard spaces.
\newblock {\em SIAM Journal on Optimization}, 24:1542--1566, 2014.

\bibitem{bacakbook}
M.~Ba\v{c}\'{a}k.
\newblock {\em Convex Analysis and Optimization in Hadamard Spaces}.
\newblock De Gruyter, Berlin, M{\"u}nchen, Boston, 2014.

\bibitem{beckfirstorder}
A.~Beck.
\newblock {\em First-Order Methods in Optimization}.
\newblock MOS-SIAM Series on Optimization. Society for Industrial and Applied
  Mathematics, 2017.

\bibitem{fenchelhadam}
G.d.C. Bento, J.C. Neto, and \'{I}.D.L. Melo.
\newblock Fenchel conjugate via {B}usemann function on {H}adamard manifolds.
\newblock {\em Applied Mathematics and Optimization}, 88(83), 2023.

\bibitem{wassersteincone}
J.~Bertrand and B.~Kloeckner.
\newblock A geometric study of {W}asserstein spaces: {H}adamard spaces.
\newblock {\em Journal of Topology and Analysis}, 04(04):515--542, 2012.

\bibitem{billeratree}
L.~J. Billera, S.~P. Holmes, and K.~Vogtmann.
\newblock Geometry of the space of phylogenetic trees.
\newblock {\em Advances in Applied Mathematics}, 27(4):733--767, 2001.

\bibitem{borweinalmost}
J.M. Borwein.
\newblock Direct theorems in semi-infinite convex programming.
\newblock {\em Mathematical Programming}, 21:301--318, 1981.

\bibitem{boumal2023intromanifolds}
N.~Boumal.
\newblock {\em An Introduction to Optimization on Smooth Manifolds}.
\newblock Cambridge University Press, 2023.

\bibitem{bridson}
M.R. Bridson and A.~Haefliger.
\newblock {\em Metric Spaces of Non-Positive Curvature}.
\newblock Springer Berlin, 1999.

\bibitem{Clarke1998}
F.~H. Clarke, Y.~S. Ledyaev, R.~J. Stern, and P.~R. Wolenski.
\newblock {\em Nonsmooth Analysis and Control Theory}, volume 178 of {\em
  Graduate Texts in Mathematics}.
\newblock Springer-Verlag, New York, 1998.

\bibitem{bettercurv}
C.~Criscitiello and N.~Boumal.
\newblock Curvature and complexity: better lower bounds for geodesically convex
  optimization.
\newblock {\em Proceedings of Machine Learning Research}, 195:1--45, 2023.

\bibitem{criscitiello2025}
C.~Criscitiello and J.~Kim.
\newblock Horospherically convex optimization on {H}adamard manifolds, part
  {I}: Analysis and algorithms, May 2025.
\newblock arXiv:2505.16970.

\bibitem{duchesne2016}
B.~Duchesne.
\newblock Groups acting on spaces of non-positive curvature, 2016.
\newblock arXiv:1603.04573.

\bibitem{fan2023horospherical}
X.~Fan, C.-H. Yang, and B.~C. Vemuri.
\newblock Horospherical decision boundaries for large margin classification in
  hyperbolic space.
\newblock In {\em Thirty-seventh Conference on Neural Information Processing
  Systems}, 2023.

\bibitem{riemsub}
O.P. Ferreira and P.R. Oliveira.
\newblock Subgradient algorithm on {R}iemannian manifolds.
\newblock {\em Journal of Optimization Theory and Applications}, 97:93--104,
  1998.

\bibitem{GigliNobili2021}
N.~Gigli and F.~Nobili.
\newblock A differential perspective on gradient flows on
  {CAT}($\kappa$)-spaces and applications.
\newblock {\em Journal of Geometric Analysis}, 31(12):11780--11818, 2021.

\bibitem{splitting-subgradient}
A.~Goodwin, A.S. Lewis, G.~L{\'o}pez-Acedo, and A.~Nicolae.
\newblock A subgradient splitting algorithm for optimization on nonpositively
  curved metric spaces, December 2024.
\newblock arXiv:2412.06730.

\bibitem{hayashi}
K.~Hayashi.
\newblock A polynomial time algorithm to compute geodesics in {CAT}(0) cubical
  complexes.
\newblock {\em Discrete Comput. Geom.}, 65(3):636--654, April 2021.

\bibitem{convanalyshad}
H.~Hirai.
\newblock Convex analysis on {H}adamard spaces and scaling problems.
\newblock {\em Foundations of Computational Mathematics}, 24:1979--2016, 2024.

\bibitem{quasiincr}
Y.~Hu, C.~K.~W. Yu, and X.~Yang.
\newblock Incremental quasi-subgradient methods for minimizing the sum of
  quasi-convex functions.
\newblock {\em Journal of Global Optimization}, 75(4):1003--1028, 2019.

\bibitem{stickylim}
S.~Huckemann, J.C. Mattingly, E.~Miller, and J.~Nolen.
\newblock Sticky central limit theorems at isolated hyperbolic planar
  singularities.
\newblock {\em Electronic Journal of Probability}, 20:1--34, 2015.

\bibitem{AHMADIKAKAVANDI}
B.~A. Kakavandi and M.~Amini.
\newblock Duality and subdifferential for convex functions on complete {CAT}(0)
  metric spaces.
\newblock {\em Nonlinear Analysis: Theory, Methods \& Applications},
  73(10):3450--3455, 2010.

\bibitem{kapovich2009}
M.~Kapovich, B.~Leeb, and J.~J. Millson.
\newblock Convex functions on symmetric spaces, side lengths of polygons and
  the stability inequalities for weighted configurations at infinity.
\newblock {\em Journal of Differential Geometry}, 81(2):297--354, 2009.

\bibitem{Knutson2001}
	A.~Knutson and T.~Tao.
\newblock Honeycombs and sums of {H}ermitian matrices.
\newblock {\em Notices of the American Mathematical Society}, 48(2):175--186, 2001.

\bibitem{sumrule}
A.~S. Lewis, G.~L\'{o}pez-Acedo, and A.~Nicolae.
\newblock Basic convex analysis in metric spaces with bounded curvature.
\newblock {\em SIAM Journal on Optimization}, 34(1):366--388, 2024.

\bibitem{lewis2024horoballs}
A.~S. Lewis, G.~L{\'o}pez-Acedo, and A.~Nicolae.
\newblock Horoballs and the subgradient method, 2024.
\newblock arXiv.2403.15749.

\bibitem{bertseknedic}
A.~Nedic and D.P. Bertsekas.
\newblock Incremental subgradient methods for nondifferentiable optimization.
\newblock {\em SIAM Journal on Optimization}, 12(1):109--138, 2001.

\bibitem{OhtaPalfia2015}
S.~Ohta and M.~P{\'a}lfia.
\newblock Discrete-time gradient flows and law of large numbers in {A}lexandrov
  spaces.
\newblock {\em Calculus of Variations and Partial Differential Equations},
  54(2):1591--1610, 2015.

\bibitem{sturmmean}
M.~Owen, E.~Miller, and J.S. Provan.
\newblock Polyhedral computational geometry for averaging metric phylogenetic
  trees.
\newblock {\em Advances in Applied Mathematics}, 68:51--91, 2015.

\bibitem{treegeo}
M.~Owen and S.~Provan.
\newblock A fast algorithm for computing geodesic distances in tree space.
\newblock {\em IEEE/ACM Transactions on Computational Geometric and
  Bioinformatics}, 8:2--13, 2011.

\bibitem{resnick}
S.I. Resnick.
\newblock {\em A Probability Path}.
\newblock Birkh\"auser Boston, MA, 2014.

\bibitem{rockafellar}
R.T. Rockafellar and R.J.B. Wets.
\newblock {\em Variational Analysis}.
\newblock Grundlehren der mathematischen Wissenschaften. Springer, 1998.

\bibitem{Schotz2025}
C.~Sch{\"o}tz.
\newblock {Variance inequalities for transformed {F}r{\'e}chet means in
  {H}adamard spaces}.
\newblock {\em Electronic Journal of Probability}, 30:1 -- 48, 2025.

\bibitem{normincr}
Q.~Shi, C.~He, and L.~Jiang.
\newblock Normalized incremental subgradient algorithm and its application.
\newblock {\em IEEE Transactions on Signal Processing}, 57(10):3759--3774,
  2009.

\bibitem{tyler1987}
D.~E. Tyler.
\newblock A distribution-free {M}-estimator of multivariate scatter.
\newblock {\em The Annals of Statistics}, 15(1):234--251, March 1987.

\bibitem{Vaiter2013}
S.~Vaiter, G.~Peyr{\'e}, and J.~M. Fadili.
\newblock Robust polyhedral regularization.
\newblock In {\em Proc.\ 10th International Conference on Sampling Theory and
  Applications (SAMPTA)}, pages 156--159, 2013.

\bibitem{weber}
A.~Weber.
\newblock {\em The Theory of the Location of Industries}.
\newblock Chicago University Press, 1929.

\bibitem{geoconvex}
H.~Zhang and S.~Sra.
\newblock First-order methods for geodesically convex optimization.
\newblock {\em Journal of Machine Learning Research}, 49:1--22, 2016.



\end{thebibliography}

\appendix
\section*{Appendix}
\section{Proof of Proposition \ref{prop:boundconv}} \label{appendix-A}
			Suppose first that $\xi_n\to \xi$ in $X^\infty$, which is to say $b_{\xi_n}\to b_{\xi}$ uniformly on 
			bounded subsets of $X$. Fix $\delta > 0$ and let $\eps > 0$ be arbitrary. By assumption, there
			exists $N \in \N$ such that
			\begin{equation}
				\label{eqn:cauchybuse}
		 	|b_{\xi_m}(r_n(\delta)) +\delta| = 
			|b_{\xi_m}(r_n(\delta)) - b_{\xi_n}(r_n(\delta))| < \frac{\eps^2}{2\delta}  \text{ for all } n,m \geq N.
		\end{equation}
		The inequality in \cite[Lemma II.8.21(2)]{bridson} asserts
		\[\frac{d(r_n(\delta),r_m(\delta))^2}{2\delta} \leq d(r_m(\delta+t),r_n(\delta)) - d(r_m(\delta+t),r_m(\delta)) 
		\text{ for all } t\geq 0.\]
		Making the substitution $s = \delta + t$, this simplifies to
		\[\frac{d(r_n(\delta),r_m(\delta))^2}{2\delta} \leq d(r_m(s),r_n(\delta)) - s + \delta  \text{ for all } s \geq \delta.\]
		Letting $s\to \infty$ we conclude
		\[\frac{d(r_n(\delta),r_m(\delta))^2}{2\delta} \leq b_{\xi_m}(r_n(\delta)) + \delta. \]
		Now use (\ref{eqn:cauchybuse}) to bound the righthand side by $\eps^2/(2\delta)$, which leads to
		\[d(r_n(\delta),r_m(\delta)) \leq \eps  \text{ for all } n,m \geq N.\]
		Since $\eps > 0$ was arbitrary, the sequence $\st{r_n(\delta)}_{n=1}^\infty$ is Cauchy and thus converges since $X$ is complete. It
		remains to prove that its limit is $r(\delta)$.
	
		{By \cite[Proposition II.8.22]{bridson}, we know that $r(\delta)$ is the unique minimizer of
		$b_\xi$ on the sphere of radius $\delta$ around $\bar x$, denoted by $S_\delta(\bar x)$.
		Thus, $\min_{z\in S_\delta(\bar x)}b_\xi(z) = b_\xi(r(\delta)) = -\delta$. Likewise, it is easy to see from
		\eqref{eqn:cauchybuse} that $\lim_{n\to \infty}b_\xi(r_n(\delta)) = -\delta$. As the limit
		$\bar z = \lim_{n\to \infty} r_n(\delta)$ is in $S_\delta(\bar x)$, we conclude that 
		\[-\delta = \min_{z\in S_\delta(\bar x)}b_\xi(z) \leq b_\xi(\bar z) = \lim_{n\to \infty}b_\xi(r_n(\delta)) = -\delta.\]
		Thus equality holds throughout, and $b_\xi(\bar z) = -\delta$. The uniqueness of $b_\xi$'s minimizer on
	$S_\delta(\bar x)$ forces $\bar z = \lim_{n\to \infty}r_n(\delta) = r(\delta)$ as desired.}

Conversely, let us suppose that $r_n(\delta)\to r(\delta)$ for all $\delta > 0$ and prove $b_{\xi_n}\to b_\xi$
		uniformly on bounded subsets. It suffices to prove $b_{\xi_n}\to b_\xi$ uniformly on each ball $B_\delta(\bar x)$
		for all $\delta > 0$. Fix $\delta > 0$ and let $\eps > 0$ be arbitrary.
The inequality in 
		\cite[Lemma II.8.21(1)]{bridson} asserts the existence of $R > 0$ such that
		\begin{equation}
			\label{eqn:trianglebuse}
			0\leq d(z,r_n(R)) + t - d(r_n(R+t),z) \leq \eps \text{ for all } z \in B_\delta(\bar x), t \geq 0, n\in \N.
		\end{equation}
		Let $z\in B_\delta(\bar x), n,m \in \N$ and $t\geq 0$. Applying first the triangle inequality and then 
		\eqref{eqn:trianglebuse}, we deduce
		\begin{align*}
			d(z,r_n(R+t)) - d(z,r_m(R+t)) &\leq d(z,r_n(R)) + t - d(z,r_m(R+t))\\
			&= d(z,r_m(R)) + t - d(z,r_m(R+t))\\
			&~ ~ ~ + d(z,r_n(R)) - d(z,r_m(R))\\
			&\leq \eps + |d(z,r_n(R)) - d(z,r_m(R))|.
		\end{align*}
		After switching the roles of $m$ and $n$ above we finally get 
		\begin{equation}
			\label{eqn:trianglebuse2}
		|d(z,r_n(R+t)) - d(z,r_m(R+t))| \leq \eps + |d(z,r_n(R)) - d(z,r_m(R))|.
\end{equation}

	Letting $t\to \infty$ in (\ref{eqn:trianglebuse2}), we deduce the following inequality for all $z\in B_\delta(\bar x)$:
	\[|b_{\xi_n}(z) - b_{\xi_m}(z)| \leq \eps + |d(z,r_n(R)) - d(z,r_m(R))| \leq \eps + d(r_n(R),r_m(R)).\]
	Taking the supremum over $z\in B_\delta(\bar x)$, it follows that 
	\[\limsup_{n,m\to \infty}\sup_{z\in B_\delta(\bar x)}|b_{\xi_n}(z)-b_{\xi_m}(z)|\leq \eps  \text{ for all } \eps > 0.\] 
	As a consequence, $\lim_{n,m\to \infty} \sup_{z\in B_\delta(\bar x)}|b_{\xi_n}(z)-b_{\xi_m}(z)| = 0$.
	That is to say $\st{b_{\xi_n}}_{n=1}^\infty$ satisfies
	the Cauchy criterion for uniform convergence on $B_\delta(\bar x)$ as desired.
	
\section{Proof of Lemma \ref{lem:convergence}}
			Suppose first that $[\xi_n,s_n] \to [0]$. Then $[\xi_n,s_n]$ is eventually 
			contained in any open neighborhood of $[0]$. Let $q\from X^\infty \times \R_+\to CX^\infty$ be the quotient map.
			For any $\eps > 0$ define the set $W_\eps = q(X^\infty\times [0,\eps))$.
			It is readily verified that $q^{-1}(W_\eps) = X^\infty \times [0,\eps)$,
				which is open in $X^\infty\times \R_+$. 
				Thus by the definition of the quotient topology on 
				$CX^\infty$, the set $W_\eps$ is open in $CX^\infty$ and contains $[0]$. 
				It follows that $[\xi_n,s_n]$ is eventually contained in $W_\eps$, which then
				implies $s_n < \eps$ eventually. Since $\eps > 0$ was arbitrary, $s_n \to 0$.

				For the converse implication when $s = 0$ we assume $X$ is proper, hence
				$X^\infty$ is compact. 
				Suppose $s_n \to 0$ and fix any neighborhood $U$ of $[0]$. 
				Then for each $\zeta \in X^\infty$ there exists an open set $W_\zeta \seq X^\infty$ and 
				$\eps_\zeta > 0$ such that
				$(\zeta,0)\in W_\zeta \times [0,\eps_\zeta) \seq q^{-1}(U)$ because $q^{-1}(U)$ is open in $X^\infty
				\times \R_+$. The sets $\st{W_\zeta \times [0,\eps_\zeta)}_{\zeta \in X^\infty}$ form an open cover
				of the compact set $X^\infty \times \st{0}$, from which we select a finite subcover
				$\st{W_{\zeta_i}\times [0,\eps_{\zeta_i})}_{i=1}^N$. Let $V :=  \bigcup_{i=1}^N W_{\zeta_i}\times [0,\eps_{\zeta_i})$. To summarize, $X^\infty \times \st{0} \seq V \seq q^{-1}(U)$.
				Since $s_n$ is eventually smaller than $\min_{i=1,\dots,N}\eps_{\zeta_i}$,
				we have $(\xi_n,s_n) \in V$ for all $n$ sufficiently large. Then $[\xi_n,s_n]\in q(V) \seq q(q^{-1}(U)) = U$ for all $n$ sufficiently large, which says $[\xi_n,s_n]$ converges to $[0]$ since the open set $U$ around $[0]$ 
				was arbitrary.

				Now suppose $[\xi_n,s_n]\to [\xi,s]$ with $s > 0$. As before, $[\xi_n,s_n]$ is eventually
				contained in any open neighborhood of $[\xi,s]$. For any $\eps > 0$ with $s - \eps > 0$, and any
				open set $U \seq X^\infty$ containing $\xi$, consider the set 
				$W_{U,\eps} = q(U \times (s-\eps, s+ \eps))$. 
				As before it is easy to check that $q^{-1}(W_{U,\eps}) = U \times (s-\eps, s+\eps)$ which is
				open in $X^\infty \times \R_+$, hence $W_{U,\eps}$ is open in $CX^\infty$. It follows
				that $[\xi_n,s_n]$ is eventually contained in $W_{U,\eps}$, which implies $\xi_n \in U$ eventually
				and $s_n \in (s-\eps,s+\eps)$ eventually. Since $U,\eps$ were arbitrary, we conclude $\xi_n\to \xi$ and
				$s_n \to s$. Conversely, suppose $\xi_n \to \xi$ and $s_n\to s > 0$. 
				If $U$ is an open neighborhood of $[\xi,s]$
				then $(\xi,s) \in q^{-1}(U)$ so there exists an open set $W \seq X^\infty$ and $\eps > 0$ such that $(\xi,s) \in W \times (s-\eps,s+\eps) \seq q^{-1}(U)$. Then $\xi_n \in W$ eventually, and $s_n \in (s-\eps,s+\eps)$ eventually, implying
				$[\xi_n,s_n] \in q(W\times (s-\eps,s+\eps))\seq q(q^{-1}(U)) = U$ for all $n$ large.

		\begin{remark}
			The converse implication in Lemma \ref{lem:convergence} can fail if $X$ is not proper due to the lack
			of compactness for $X^\infty$. Indeed, if $X$ is an infinite-dimensional Hilbert space
			then $X^\infty$ is homeomorphic to the unit sphere $S\seq X$ with the norm topology.
			But passing to the quotient 
			topology introduces undesirable open sets. Choose any continuous function $f\from S \to (0,+\infty)$ with infimal
			value 0 (the existence of such a function hinges on the noncompactness of $S$), and let $\st{\xi_n}_{n=1}^\infty$ be an infimizing sequence. The set 
			$V = \st{(\xi,s) \in X^\infty \times \R_+ \mid s < f(\xi)}$ is open in $X^\infty \times \R_+$ because 
			$f$ is continuous. Moreover, $X^\infty \times \st{0} \seq V$ and $q^{-1}(q(V)) = V$. Thus $q(V)$ is an open neighborhood
			of $[0]$ in $CX^\infty$, we have $f(\xi_n) \to 0$, but $[\xi_n,f(\xi_n)] \notin q(V)$ for any $n\in \N$
			and so does not converge to $[0]$.
		\end{remark}

\section{Proof of Example \ref{ex:euclid}}

			We will show that if $C\seq \R^n$ is convex and
			$f\from C \to \R$ is convex and locally Lipschitz on $C$ then $f$ has a (Euclidean) subgradient at each point 
			 in $C$. 
			For any $x\in C$, if $f$ is locally 
			Lipschitz one can find $L, \delta > 0$ such that 
			$|f(y) - f(z)| \leq L \|y-z\|$ for all $y\in C \cap B_{\delta}(x) =: C'$.
			Define $g\from \R^n \to (-\infty,+\infty]$ as follows: 
			\[ g(y) = \begin{cases}
					f(y), & y\in C'\\
					+\infty, & \text{else}.
			\end{cases}\]
			Then set $h(y) = \inf_{z\in \R^n} \st{g(z) + L\|y-z\|}$. We claim that 
			$h$ is finite-valued, convex, and agrees with $f$ on $C'$. Clearly $h(y) < +\infty$ for any $y$ because
			$C'$ is nonempty, while $h(y) > -\infty$ because $C'$ is bounded and
			\[h(y) = \inf_{z\in C'} \st{f(z) + L\|y-z\|} \geq f(x) + L\inf_{z\in C'}\st{\|y-z\|-\|z-x\|} > -\infty.\]
			Thus $h$ is finite-valued, and convexity of $h$ follows by recognizing $h$ as the 
			infimal convolution of the convex function $g$ with the finite-valued convex function 
			$L\|\cdot\|$. Given $y\in C'$ we observe $h(y) \leq f(y)$ because $y$ is feasible for the infimum defining
			$h$, while on the other hand the $L$-Lipschitz property for $f$ on $C'$ implies 
			\[f(z) + L\|y-z\| \geq f(y) \text{ for all } z\in C'.\]
			Taking the infimum over $z\in C'$ implies $h(y) \geq f(y)$, hence $h \equiv f$ on $C'$.

			Now, $h$ is a real-valued convex function on $\R^n$ and thus admits a subgradient $v$ at $x$. 
			We claim that $v$ is a subgradient of $f$.
			The subgradient inequality for $h$ reads
			\[h(y)\geq h(x) + v^T(y-x) \text{ for all } y\in \R^n.\]
			Restricting to points in $C'$ where $h \equiv f$ this shows 
			$\tilde{f} := f-v^T(\cdot- x)$ is minimized
			over $C'$ at $x$. By definition of $C'$ it follows that $x$ is a local minimizer of $\tilde f$ on $C$,
			from which convexity of $\tilde f$ implies $x$ minimizes $\tilde f$ on $C$. This says that 
			$v$ is a subgradient of $f$ at $x$.
		
\section{Proof of Example \ref{ex:convcont}}
					If $f \from \R^n \to \R$ is convex, then Fenchel biconjugation allows us to write it in the form
				$f(x) = \sup\st{y^Tx - f^*(y) \from y\in \R^n}$ where $f^* \from \R^n \to \clos \R$ is the convex conjugate of $f$.
				The supremum is always attained by choosing any subgradient $y \in \p f(x)$, 
				which is always nonempty since $f$ 
				is continuous. Furthermore, our earlier remarks show that we can identify
				$C(\R^n)^\infty$ with $\R^n$ itself via the bijection $\varphi \from C(\R^n)^\infty \to \R^n$,
				$\varphi([\xi,s]) = -s\xi$ (viewing $\xi$ as an element of $\mathbb{S}^{n-1}$). If we 
					adopt the reference point $\bar x = 0$ then our Busemann functions take the form
					$b_\xi(z) = -\xi^Tz$. Then we derive:
					\begin{align*}
						f(x) &= \max \st{y^Tx - f^*(y) \from y\in \R^n}\\
						&= \max\st{\varphi([\xi,s])^Tx - (f^*\circ \varphi)([\xi,s]) \from [\xi,s] \in C(\R^n)^\infty}\\
						&= \max\st{-s\xi^Tx - (f^*\circ \varphi)([\xi,s]) \from [\xi,s] \in C(\R^n)^\infty}\\
						&= \max\st{\ip{x,[\xi,s]} - (f^*\circ \varphi)([\xi,s]) \from [\xi,s] \in C(\R^n)^\infty}.
					\end{align*}
					Defining $g := f^* \circ\varphi$, we see that $f = g^\circ$ and since the supremum is always
					attained $f$ is a Busemann envelope.
				Thus we recover the class of real-valued convex functions on $\R^n$. More generally, 
				this argument works for continuous convex functions on a Hilbert space.

\section{Proof of Example \ref{ex:dist}}
			
			 Let $\tilde{r}(t) = r_{x,a}(d(x,a) + t)$ be a ray issuing
			from $a$ in direction $\xi_x$, so $b_{a,\xi_x} = b_{\tilde{r}}$. Then 
			\begin{equation}
				\label{eqn:busedist}
				b_{a,\xi_x}(x) = b_{\tilde{r}}(x) = \lim_{t\to \infty}(d(x,r_{x,a}(d(x,a)+t)) - t) = d(x,a).
		\end{equation}
			Now write:	
			\begin{align}
				\sup\st{\ip{x,[\xi,s]} - g([\xi,s]) \from [\xi,s] \in CX^\infty} &=
				\sup\st{ \ip{x,[\xi,s]} - g([\xi,s]) \from [\xi,s] \in K} \nonumber \\
				&= \sup\st{\ip{x,[\xi_z,1]} - b_{\xi_z}(a) \from z\in X\setminus \st{a}} \nonumber\\
				&= \sup\st{b_{\xi_z}(x) - b_{\xi_z}(a) \from z\in X\setminus \st{a}} \label{eqn:distalign}.
			\end{align} 
			If $x = a$ then this last expression is identically zero, so choose any $z\in X\setminus \st{a}$ and observe
			$[\xi_z,1]\in CX^\infty$ attains the supremum in (\ref{eqn:distalign}).
			We have the upper bound 
			$b_{\xi_z}(x) - b_{\xi_z}(a) \leq d(x,a)$ for all $z\in X\setminus \st{a}$ because Busemann 
			functions are 1-Lipschitz. On the other hand, if $x\neq a$ then choosing 
			$z = x$ we use Proposition \ref{additive} and  (\ref{eqn:busedist})
			to conclude 
			\[b_{\xi_x}(x) - b_{\xi_x}(a) = b_{a,\xi_x}(x) = d(x,a).\] 
			Thus the supremum in (\ref{eqn:distalign}) is attained by $[\xi_x,1] \in CX^\infty$.
			It follows that for all $x\in X$ the supremum in the first line above is always attained, 
			allowing us to say 
			\[d(x,a) = \max\st{\ip{x,[\xi,s]} - g([\xi,s]) \from [\xi,s] \in CX^\infty} \]
			is Busemann subdifferentiable by Proposition \ref{ex:busemax}.

\section{Proof of Example \ref{ex:distballs}}
If $x\in H_r$ then $\dist(x,H_r) = 0 = \max\st{0,b_r(x)}$ so it remains to consider $x\notin H_r$.
		We use the Lipschitz property of $b_r$ and the definition of $H_r$ to deduce
		\[b_r(x) \leq b_r(y) + d(x,y) \leq d(x,y) \text{ for all } y\in H_r. \]
		Taking the infimum over $y\in H_r$ implies $b_r(x) \leq \dist(x,H_r)$.
		On the other hand, for any $t\geq 0$ one has $B_t(r(t)) \seq H_r$ because 
		$s\mapsto d(x,r(s)) - s$ is nonincreasing.
			As a consequence, 
		$\dist(x,H_r) \leq \dist(x,B_t(r(t)))$. But since $x\notin H_r$ we must have
		$x\notin B_t(r(t))$ for any $t\geq 0$, so
		$\dist(x,B_t(r(t)) = d(x,r(t))- t$. To summarize,
			\[\dist(x,H_r)\leq d(x,r(t)) - t \text{ for all } t\geq 0.\]
			As $t\to \infty$ we get $\dist(x,H_r) \leq b_r(x)$, proving $\dist(x,H_r) = b_r(x)$
			for $x\notin H_r$.
			
\section{Proof of Proposition \ref{biconjugate}}
		The definition of the biconjugate implies $f^{\bullet\circ} \leq f$. Assuming $f$ is Busemann subdifferentiable,
		choosing a Busemann subgradient $[\xi,s]$ at a point $x\in X$ implies via \eqref{eqn:prefench} that
		\[f(x) = \ip{x,[\xi,s]}- f^\bullet([\xi,s]) \leq f^{\bullet\circ}(x).\]
		Thus $f = f^{\bullet\circ}$.
		Suppose conversely that $f = f^{\bullet\circ}$, $f$ is continuous, and $X$ is proper, so $X^\infty$ is compact. Then one has
		\[f(x) = \sup_{[\xi,s]\in CX^\infty}{\ip{x,[\xi,s]} - f^\bullet([\xi,s])}.\]
		The conjugate $f^\bullet$ is lower semicontinuous as a supremum of continuous functions, so the function
		$\phi_x([\xi,s]) = \ip{x,[\xi,s]} - f^\bullet([\xi,s])$ is upper semicontinuous. Choose a maximizing sequence
		$\st{[\xi_n,s_n]}_{n=1}^\infty$, meaning 
		$f(x) = \lim_{n\to \infty} \ip{x,[\xi_n,s_n]} - f^\bullet([\xi_n,s_n])$.
		The sequence $\st{\xi_n}_{n=1}^\infty$ has a convergent subsequence because $X^\infty$ is compact, so without loss
		of generality say $\xi_n\to \xi$. If $\st{s_n}_{n=1}^\infty$ is unbounded
		then after extracting a subsequence we may assume $s_n\to \infty$.
		Then for each $R > 0$, letting $\bar x$ be the underlying basepoint for all Busemann functions,
		\begin{align*}
			\frac{f^\bullet([\xi_n,s_n])}{s_n} &\geq \sup_{y\in B_R(\bar x)}\st{\ip{y,[\xi_n,1]} - \frac{f(y)}{s_n}}\\
			&\geq \sup_{y\in B_R(\bar x)} b_\xi(y) - \sup_{y\in B_R(\bar x)}\frac{f(y)}{s_n}\\
			&= R - \frac{\sup_{y\in B_R(\bar x)} f(y)}{s_n}. 
		\end{align*}
		The supremum of $f$ on $B_R(\bar x)$ is finite because $f$ is continuous. 
		Letting $n\to \infty$ reveals 
		\[\liminf_{n\to \infty} \frac{f^\bullet([\xi_n,s_n])}{s_n} \geq R.\]
		Since $R > 0$ was arbitrary, $\lim_{n\to \infty} \frac{f^\bullet([\xi_n,s_n])}{s_n} = +\infty$.
		Then
		\[f(x) = 
		\lim_{n\to \infty} s_n\left( b_{\xi_n}(x) - \frac{f^\bullet([\xi_n,s_n])}{s_n}\right)
		\leq \lim_{n\to \infty} s_n \left(d(x,\bar x) - \frac{f^\bullet([\xi_n,s_n])}{s_n}\right) = -\infty.
	\]
		This contradicts the finite-valuedness of $f$, so $\st{s_n}_{n=1}^\infty$ is bounded. Extracting a further 
		subsequence, we may assume $s_n \to s\geq 0, \xi_n \to \xi$. Upper semicontinuity of 
		$\phi_x$ ensures
		\[\ip{x,[\xi,s]} - f^\bullet([\xi,s]) = \phi_x([\xi,s]) \geq \limsup_{n\to \infty} \phi_x([\xi_n,s_n]) = f(x).\]
		Thus equality holds in the Fenchel-Young inequality \eqref{eqn:prefench2}, meaning $[\xi,s]$ is a Busemann
		subgradient for $f$ at $x$. Since $x$ was arbitrary, we are done.

\section{Proof of Proposition \ref{prop:horoepi}}

To prove the implication (ii) $\Rightarrow$ (i), we suppose that $f$ has a horospherically convex epigraph. Let $x \in X$ and consider the point
		$(x,f(x))$ in the boundary of $\epi f$. Since $\epi f$ is horospherically convex, there exists
		a ray $r\from \R_+ \to X\times \R$ issuing from $(x,f(x))$ such that $\epi f \seq b_r^{-1}((-\infty,0])$.
		The point at infinity $r(\infty)$ lies in $(X\times \R)^\infty$, which is homeomorphic to the 
		spherical join of $X^\infty$ and $\R^\infty$ (see \cite[Example II.8.11(6)]{bridson}), so we may assume
		$b_r(z,t) = \alpha b_{r_1}(z) + \beta b_{r_2}(t)$ for rays $r_1\from \R_+\to X,r_2\from \R_+ \to \R$
		and scalars $\alpha,\beta \geq 0$ satisfying $\alpha^2+\beta^2 =1$. 
		Realizing that $r_2$ must be of the form $r_2(t) = f(x) \pm t$, we derive the following inequality
		for all $(y,s) \in \epi f$:
		\begin{equation}
			\label{eqn:busesplit}
		0 \geq b_r(y,s) = \alpha b_{r_1}(y) \pm \beta (f(x) - s).
	\end{equation}
		If $\beta = 0$, then 
		$b_{r_1}(y) \leq 0$ for all $y\in X$. Since $X$ has the geodesic extension property, we can
		extend $r_1$ into a geodesic line $r_1'$, and obtain a contradiction:
		\[0\geq b_{r_1}(r_1'(-1)) = 1.\]
		If the coefficient of $\beta > 0$ is $-1$, then taking $y = x, s = f(x) + \lambda$ for $\lambda > 0$ one eventually
		obtains a contradiction for $\lambda$ sufficiently large. Thus the coefficient of $\beta$ must be $+1$.
		Finally, taking $s = f(y)$ for each $y\in X$ we observe directly from \eqref{eqn:busesplit} 
		that $f$ has Busemann subgradient
		$[r_1(\infty),\alpha/\beta]$ at $x$.  We have proved property (i).
				
		Next, to prove the implication (iii) $\Rightarrow$ (ii), suppose $f$ is Busemann subdifferentiable and continuous. Let $(x,t)$ lie in the boundary of $\epi f$. 
		By continuity of $f$, $t = f(x)$.
		Busemann subdifferentiability of $f$ at $x$ yields a Busemann subgradient $[\xi,s]$ such that
		\[f(y) \geq f(x) + sb_{x,\xi}(y)\]
		for all $y\in X$. Equivalently,
		\[ 0\geq f(x) - t + sb_{x,\xi}(y) \text{ for all $(y,t) \in \epi f$ }. \]
		Normalizing by a factor of $(1+s^2)^{-1/2}$, this becomes
		\[0 \geq \frac{1}{\sqrt{1+s^2}}(f(x) - t) + \frac{s}{\sqrt{1+s^2}}b_{x,\xi}(y) \text{ for all $(y,t) \in \epi f$ }.\]
		This is exactly the Busemann function for the ray 
		\[r(t) = \left(r_{x,\xi}\left(\frac{s}{\sqrt{1+s^2}}t\right),f(x) + \frac{1}{\sqrt{1+s^2}}t\right),\]
		which supports $\epi f$ at $(x,f(x))$ by the inequality above.  This completes our proof of property (ii).
		
Now suppose that the space $X$ is proper. We first show the implication (i) $\Rightarrow$ (ii).	
By Proposition \ref{biconjugate}, 
		\[f(x) = f^{\bullet\circ}(x) = \sup_{[\xi,s]\in CX^\infty}\st{\ip{x,[\xi,s]} - f^\bullet([\xi,s])}.\]
		Thus $f$ is a supremum of nonnegatively scaled Busemann functions, whose epigraphs are 
		easily seen to be horoballs in $X\times \R$. 
		The epigraph of $f$ is therefore an intersection of horoballs, and such intersections are horospherically convex
		when $X$ is proper: see 
		\cite[Theorem 10]{criscitiello2025} for a proof in the manifold case which generalizes, nearly verbatim, 
		to proper Hadamard spaces. This proves (ii).

Finally, assuming that the space $X$ is proper, we prove the implication (i) $\Rightarrow$ (iii).  Any Busemann subdifferentiable function is automatically lower semicontinuous, so it suffices to prove that for each $\bar \alpha \in \R$ the strict level set 
		 $L = \st{x\in X\mid f(x) < \bar \alpha}$ is open. Consider any point $\bar x\in L$.
		Then $(\bar x,\bar\alpha) \in \epi f$, and if 
		$(\bar x,\bar \alpha)$ lies in $\inter (\epi f)$ then there exists $\eps > 0$ such that 
		$f(x)\leq \alpha$ whenever $d(x,\bar x)^2 + (\alpha - \bar \alpha)^2 < \eps^2$. 
		This ensures that if $d(x,\bar x) < \eps$ then $f(x) < \bar \alpha$, so 
		$\bar x$ lies in the interior of $L$. On the other hand, if $(\bar x,\bar \alpha)$ is not in the interior of $\epi f$ then it must lie on the boundary.  Property (ii) implies the existence of
a supporting horoball to 
$\epi f$ at $(\bar x,\bar \alpha)$. As in the proof of (ii) $\Rightarrow$ (i), we deduce the 
		existence of a ray $r\from \R_+\to X$ issuing from $\bar x$
		and scalars $\rho\geq 0, \sigma > 0$ such that $\rho^2 + \sigma^2 = 1$
		and
		\[\rho b_r(y) + \sigma (\bar \alpha - s) \leq 0 ~ \text{ for all } (y,s) \in \epi f.\]
		Noting $(\bar x, f(\bar x))\in \epi f$ yields the contradiction
		$\bar \alpha \leq f(\bar x)$, completing the proof.

\section{Proof of Example \ref{ex:distnotbuse2}}
			Let $\bar x = (0,0,0)$ be the reference point with respect to which we define our Busemann functions.
			Suppose for a contradiction that $f$ has a Busemann subgradient $[\xi,s]$ at $x_0 = (1,1,0)$. 
			As $x_0$ is not a minimizer of $f$, we have $s > 0$. Moreover,
			\begin{equation}
				\label{eqn:subatxzero}
				f(y) - sb_\xi(y) \geq \frac{\sqrt{2}}{2} - sb_\xi(x_0) \text{ for all } y\in X.
			\end{equation}

			\textbf{Case 1:} Suppose $\xi$ is the direction of the ray $\gamma$ issuing from $\bar x$ and passing
			through $x_0$. Taking $y = (1/2,1/2,0)$ we have $f(y) = 0, b_\xi(y) = -\sqrt{2}/2$, and $b_\xi(x_0) = -\sqrt{2}$, in
			contradiction to \eqref{eqn:subatxzero}.
			
			\textbf{Case 2:} Suppose $\xi$ is the direction of a ray $r$ in the plane $\st{z = 0}$ issuing from $\bar x$ and obtained
			by moving $\gamma$ an angle $\alpha\in (0,3\pi/4]$ in the trigonometric sense. Taking $y = (3/2,1/2,0)$, we have
			$f(y) = \sqrt{2}/2$ and $d(\bar x,y) = \sqrt{5}/\sqrt{2}$. Moreover, for all $t > 0$,
			\[					d(x_0,r(t))^2 = 2 + t^2 -2\sqrt{2}t\cos\alpha
			\]			and
\[				d(y,r(t))^2 = \frac{5}{2} + t^2 -2t(\sqrt{2}\cos\alpha- (1/\sqrt{2})\sin\alpha).\]
Hence, $b_\xi(x_0) = -\sqrt{2}\cos\alpha$ and $b_\xi(y) = -\sqrt{2}\cos\alpha + (1/\sqrt{2})\sin\alpha$, in contradiction to \eqref{eqn:subatxzero}.
			
\textbf{Case 3:} Suppose that  $\xi$ is the direction of a ray $r$ in the quadrant $\st{x\leq 0, y = 0, z\geq 0}$. Taking $y=(3/2,1/2,0)$,
we have $f(y) = \sqrt{2}/2$ and $b_\xi(x_0) \leq \sqrt{2}$.

Fix $t > 0$ and let $v = (-s,0,0)$, where $s\geq 0$, such that $d(y,r(t)) = d(y,v) + d(v,r(t))$. Then $s\leq t$
and
\[d(y,r(t)) = d(v,r(t)) + \sqrt{\frac{5}{2}+s^2+3s} \geq t-s + \sqrt{\frac{5}{2}+s^2+3s} \geq \sqrt{\frac{5}{2}+t^2+3t}.\]
Therefore $b_\xi(y) \geq 3/2 > b_\xi(x_0)$, a contradiction. The remaining cases follow by symmetry.

	
\section{Proof of Example \ref{ex:nonhoro}}
		
			Defining $\tilde{x}= (2/5,0,0), \bar x = (1/4,1/4,0)$, we do some preliminary calculations:
			\begin{equation}
\label{eqn:prelim}
f(\tilde{x}) = 37/25, ~ ~ ~ f(\bar{x}) = 13/8.
	\end{equation}
			Any ray issuing from $\bar x$ is determined by a choice of 
			unit vector $v = (v_1,v_2,0)$, or equivalently an angle $\theta\in[0,2\pi]$ such that 
			$v = (\cos \theta,\sin \theta,0)$. By symmetry it suffices to consider $\pi/4 \leq \theta \leq 5\pi/4$.
			We denote a ray $r$ originating from $\bar x$ in direction $v$ by $r_v$.
		
		\textbf{Case 1:} 
		Corresponding to the case $\theta = 5\pi/4$, consider the direction $v= (-1/\sqrt{2},-1/\sqrt{2},0)$,
		pointing from $\bar x$ to the origin $\mathbf{0} := (0,0,0)$. Since the ray ultimately extends up the spine
		$\st{0}\times \st{0}\times
		\R_+$, for $t > 0$ sufficiently large the geodesic between $r_v(t)$ and any $x\in \R_+\times\R_+ \times\st{0}$ 
		consists of the two segments $[x,\mathbf{0}] \cup [\mathbf{0},r_v(t)]$. Hence
		\[d(x,r_v(t)) = \|x\|+ \|(t-\|\bar x\|)(0,0,1)\| = \|x\| + t- \|\bar x\| \text{ for all } t \geq \|\bar x\|.\]
	The corresponding Busemann function at such an $x$ is thus
	\[b_{r_v}(x) = \|x\| - \|\bar x\|.\]
		Plugging in $\tilde x$ we find $b_{r_v}(\tilde x) = 2/5 - \sqrt{2}/4 = (8-5\sqrt{2})/20 > 0$. On the other hand, the preliminary calculations (\ref{eqn:prelim}) show $f(\tilde{x}) < f(\bar x)$, so $\tilde{x}\in f_{\bar x}\setminus \st{b_{r_v}\leq 0}$
		i.e.\ this horoball does not contain the given level set.
		
		\textbf{Case 2:} Now consider any direction $v$ corresponding to $\pi < \theta < 5\pi/4$. After sufficient time,
		the point $r_v(t)$ will inhabit the upright quadrant
		$\R_-\times\st{0}\times\R_+$. By imagining the quadrant $\R_-\times\st{0}\times\R_+$ being folded down into 
		$\R_- \times \R_-\times\st{0}$, we see that for $t > 0$ large enough  the distance $d(\mathbf{0},r_v(t))$ is equal to the Euclidean distance $\|\bar x + tv\|$ when we identify $\bar x,v$ with vectors in $\R^3$.
		Since the geodesic from $\tilde{x}$ to such an $r_v(t)$ consists again of the 
		two segments $[\tilde{x},\mathbf{0}]\cup[\mathbf{0},r_v(t)]$,
		it follows that
		\[d(\tilde{x},r_v(t)) = \|\tilde{x}\| + \|\bar x + tv\|.\]
		The corresponding Busemann function evaluated at $\tilde{x}$ is thus 
		\[b_{r_v}(\tilde{x}) = \|\tilde{x}\| + v^T\bar x \geq \|\tilde{x}\| - \|\bar x\| > 0.\]
		As in the previous case we conclude $\tilde{x}\in f_{\bar x}\setminus \st{b_{r_v}\leq 0}$.
		
		\textbf{Case 3:} Finally, consider directions $v$ corresponding to $\pi/4 \leq \theta \leq \pi$. Then we 
		have $d(\tilde{x}, r_v(t)) = \|\tilde{x} - \bar x - tv\|$, hence
		$b_{r_v}(\tilde{x}) = v^T(\bar x - \tilde{x})$. Using $v = (\cos\theta,\sin\theta,0)$ we find:
		\[b_{r_v}(\tilde{x}) =  -\frac{3}{20} \cos\theta + \frac{1}{4}\sin \theta.\]   
		Using calculus, one can show that the righthand side is a concave function of $\theta$ on $[\pi/4,\pi]$ and so attains its minimum at an endpoint. It is easy to check that the values at these endpoints are both positive, so $b_{r_v}(\tilde{x}) > 0$ for all such $\theta$. As before, $\tilde{x}\in f_{\bar x} \setminus \st{b_{r_v}\leq 0}$.
		We conclude that no ray issuing from $\bar x$, which lies in the boundary of $f_{\bar x}$, supports $f_{\bar x}$, so $f_{\bar x}$ is not horospherically convex.

		Define the geodesic $\eta \from [0,1]\to X, \eta(t) = \bar x + (t/\sqrt{2})(1,1,0)$.
		We will show that $([\eta],3/\sqrt{2}) \in T_{\bar x} X$ is a subgradient for $f$ at $\bar x$. 
		According to \eqref{eqn:cat0subgrad} and the value computed for $f(\bar x)$ in \eqref{eqn:prelim},
		it suffices to prove 
		\begin{equation}
			\label{eqn:finalsubgrad}
		\frac{3}{\sqrt{2}}d(\bar x,y)\cos \angle (\eta,\gamma_y) + \frac{13}{8} \leq f(y) \text{ for all } y\in X.
	\end{equation}
		Despite the initial requirement that this inequality holds for all $y\in X$, 
		it is shown in \cite[Remark 3.3(iv)]{sumrule} that the subdifferential of a geodesically convex function 
		depends only on the function locally. Thus it suffices to verify \eqref{eqn:finalsubgrad} on a small ball 
		around $\bar x$ contained in the quadrant $Q = \R_+\times\R_+\times\st{0}$, say $B_{1/5}(\bar x)$.
		For $y = (y_1,y_2,0) \in Q$ we have
		\begin{equation}
			\label{eqn:fonQ}
		f(y) = \frac{1}{2}\left( (y_1+1)^2+y_2^2 + y_1^2 + (y_2+1)^2 \right).
		\end{equation}
		Furthermore, for $y \in B_{1/5}(\bar x)$ the geodesic segment $[\bar x,y]$ is 
		the Euclidean line segment joining these points in $Q$. It follows that $d(\bar x,y) = \|y-\bar x\|$ and
		\begin{equation}
			\label{eqn:anglexpr}
			\cos \angle(\eta,\gamma_y) =  \frac{(1,1,0)^T(y-\bar x)}{\sqrt{2}\|y-\bar x\|} = 
			\frac{y_1+y_2 -1/2}{\sqrt{2}\|y-\bar x\|}.
	\end{equation}
	Thus after restricting to $y\in B_{1/5}(\bar x)$ and substituting \eqref{eqn:fonQ} and \eqref{eqn:anglexpr}
	into \eqref{eqn:finalsubgrad}, our desired inequality becomes
		\[\frac{3}{2}y_1 + \frac{3}{2}y_2 + \frac{7}{8} \leq \frac{1}{2}\left( (y_1+1)^2+y_2^2 + y_1^2 + (y_2+1)^2 \right)
		\text{ for all } y = (y_1,y_2,0) \in B_{1/5}(\bar x).\]
		This now holds by the Euclidean subgradient inequality for the differentiable convex function
		$h(a,b) = \frac{1}{2}\left( (a+1)^2+b^2 + a^2 + (b+1)^2 \right)$ at the point $(1/4,1/4)$.

\end{document}